\documentclass{article}
\usepackage{fullpage,amssymb,amsmath,amsthm,amsfonts,algorithm,algorithmic,graphicx, subfigure}

\def\lt{\left}
\def\rt{\right}

\newcommand{\bmat}{\left[ \begin{array}}
\newcommand{\emat}{\end{array} \right]}
\newcommand{\diag}{{\rm diag}}

\newcommand{\dis}{\displaystyle}
\newcommand{\ra}{\rightarrow}

\newcommand{\ignore}[1]{}
\newtheorem{theorem}{Theorem}[section]

\newtheorem{lemma}[theorem]{Lemma}
\newtheorem{corollary}[theorem]{Corollary}

\theoremstyle{definition}
\newtheorem{definition}[theorem]{Definition}

\newtheorem{remark}[theorem]{Remark}

\graphicspath{{figures/}}

\title{Minimizing Communication for Eigenproblems and the Singular Value Decomposition}

\author{Grey Ballard\thanks{CS Division, University of California, Berkeley, CA 94720},
James Demmel\thanks{Mathematics Department and CS Division,
University of California, Berkeley, CA 94720.},
and Ioana Dumitriu\thanks{Mathematics Department, University of Washington, Seattle, WA 98195.}}

\begin{document}

\maketitle

\begin{abstract}

Algorithms have two costs: arithmetic and communication. The latter represents the cost of moving data, 
either between levels of a memory hierarchy, or between processors over a
network. Communication often dominates arithmetic and represents a rapidly increasing proportion of the total cost,
so we seek algorithms that minimize communication. In \cite{BDHS10} lower bounds
were presented on the amount of communication required for essentially all
$O(n^3)$-like algorithms for linear algebra, including eigenvalue problems
and the SVD. Conventional algorithms, including those currently implemented in (Sca)LAPACK, perform asymptotically more communication
than these lower bounds require. In this paper we present parallel and sequential
eigenvalue algorithms
(for pencils, nonsymmetric matrices, and symmetric matrices) and SVD algorithms
that do attain these lower bounds, and analyze their convergence and communication
costs.

\end{abstract}

\section{Introduction} \label{Intro} 

The running time of an algorithm depends on two factors: the number of floating point operations executed (\emph{arithmetic}) and the amount of data moved either between levels of a 
memory hierarchy in the case of sequential computing, or over a 
network connecting processors in the case of parallel computing (\emph{communication}). 
The simplest metric of communication is to count the 
total number of words moved (also called the {\em bandwidth cost}).  Another simple metric is to count the number of messages containing these words (also known as the \emph{latency cost}).   
On current hardware the cost of moving a single word, or that of sending a single message, already 
greatly exceed the cost of one arithmetic operation, 
and technology trends indicate that this processor-memory gap is growing 
exponentially over time. So it is of interest to devise 
algorithms that minimize communication, sometimes even 
at the price of doing more arithmetic. 

In this paper we present sequential and parallel algorithms 
for finding eigendecompositions  and SVDs of dense matrices, 
that do $O(n^3)$ arithmetic operations, are numerically stable, 
and do asymptotically less communication than previous such algorithms. 

In fact, these algorithms attain known communication lower bounds 
that apply to many $O(n^3)$ algorithms in dense linear algebra.  
In the sequential case, when the $n$-by-$n$ input matrix does not 
fit in fast memory of size $M$, the number of words moved between 
fast (small) and slow (large) memory is $\Omega (n^3/\sqrt{M})$.
In the case of $P$ parallel processors, where each processor has 
room in memory for $1/P$-th of the input matrix, the number of 
words moved between one processor and the others is 
$\Omega (n^2 / \sqrt{P} )$. These lower bounds were originally 
proven for sequential \cite{HongKung81} and parallel \cite{ITT04} 
matrix multiplication, and extended to many other linear algebra 
algorithms in \cite{BDHS10}, including the first phase of 
conventional eigenvalue/SVD algorithms: reduction to Hessenberg, 
tridiagonal and bidiagonal forms.

Most of our algorithms, however, do not rely on reduction to these 
condensed forms; instead they rely on explicit QR factorization 
(which is also subject to these lower bounds). This raises the 
question as to whether there is a communication lower bound 
{\em independent of algorithm} for solving the eigenproblem. 
We provide a partial answer by reducing the QR decomposition
of a particular block matrix to computing the Schur form
of a similar matrix, so that sufficiently general lower bounds 
on QR decomposition could provide lower bounds for computing
the Schur form.

We note that there are communication lower bounds not just for the bandwidth cost, but also for the latency cost. 
Some of our algorithms also attain these bounds.

In more detail, we present three kinds of communication-minimizing
algorithms, {\em randomized spectral divide-and-conquer}, {\em eigenvectors from Schur Form}, 
and {\em successive band reduction}. 

Randomized spectral divide-and-conquer applies to eigenvalue problems for regular pencils
$A - \lambda B$, the nonsymmetric and symmetric eigenvalue problem of a single matrix,
and the singular value decomposition (SVD) of a single matrix. For $A - \lambda B$
it computes the generalized Schur form, and for a single nonsymmetric matrix
it computes the Schur form. For a symmetric problem or SVD, it computes the
full decomposition.
There is an extensive literature on such divide-and-conquer algorithms:
the PRISM project algorithm 
\cite{auslandertsao,bischofledermansuntsao}, 
the reduction of the symmetric eigenproblem to matrix multiplication by Yau and Lu \cite{luyau}, 
the matrix sign function algorithm originally formulated in \cite{howland,Roberts},
and the inverse-free approach originally formulated in
\cite{godunov86,bulgakov88,malyshev89,malyshev92,malyshev93},
which led to the version \cite{baidemmelgu94,DDH07} 
from which we start here. Our algorithms will minimize communication,
both bandwidth cost and latency cost, in both the two-level sequential memory model and
parallel model (asymptotically, up to polylog factors).  
Because there exist cache-oblivious sequential algorithms for matrix multiplication \cite{FLPR99} and QR decomposition \cite{FrensWise03}, we conjecture that it is possible to devise cache-oblivious versions of the randomized spectral divide-and-conquer algorithms, thereby minimizing communication costs between multiple levels of the memory hierarchy on a sequential machine.

In particular, we fix a limitation in \cite{baidemmelgu94} (also inherited
by \cite{DDH07}) which needs to compute a rank-revealing
factorization of a product of matrices $A^{-1}B$ without explicitly
forming either $A^{-1}$ or $A^{-1}B$. The previous algorithm used
column pivoting that depended only on $B$, not $A$; this cannot \emph{a priori}
guarantee a rank-revealing decomposition, as the improved version we present in this paper does. In fact,
we show here how to compute a randomized rank-revealing decomposition 
of an arbitrary product $A_1^{\pm 1} \cdots A_p^{\pm 1}$ without 
any general products or inverses, which is of independent interest.

We also improve the old methods (\cite{baidemmelgu94}, \cite{DDH07}) by providing a ``high level'' randomized 
strategy (explaining how to choose where to split the spectrum), in addition 
to the ``basic splitting'' method (explaining how to divide the spectrum 
once a choice has been made). We also provide a careful analysis of 
both ``high level'' and ``basic splitting'' strategies 
(Section \ref{Algs}), and illustrate the latter on a few numerical examples.

Additionally, unlike previous methods, we show how to deal effectively with 
clustered eigenvalues (by outputting a convex polygon, in the nonsymmetric case, 
or a small interval, in the symmetric one, where the eigenvalues lie), rather than 
assume that eigenvalues can be spread apart by scaling (as in \cite{luyau}). 

Given the (generalized) Schur form, our second algorithm computes the eigenvectors, minimizing communication by using a natural blocking scheme. Again, it works sequentially or in parallel, minimizing bandwidth cost and latency cost.  

Though they minimize communication asymptotically, 
randomized spectral divide-and-conquer algorithms
perform several times as much arithmetic
as do conventional algorithms.  For the case of regular pencils and 
the nonsymmetric eigenproblem, we know of no communication-minimizing 
algorithms that also do about the same number of floating point operations as conventional algorithms.
But for the symmetric eigenproblem (or SVD), it is possible to 
compute just the eigenvalues (or singular values) with very little extra arithmetic,
or the eigenvectors (or singular vectors) 
with just 2x the arithmetic cost for sufficiently large matrices
(or 2.6x for somewhat smaller matrices)
while minimizing bandwidth cost in the sequential case.
Minimizing bandwidth costs of these algorithms in the parallel case 
and minimizing latency in either case are open problems.  
These algorithms are a special case of the class of
{\em successive band reduction} algorithms introduced in \cite{SBR1,SBR2}.

The rest of this paper is organized as follows.
Section~\ref{rrdr} discusses randomized rank-revealing decompositions.
Section~\ref{Algs} uses these decompositions to implement randomized spectral divide-and-conquer algorithms.
Section~\ref{cba} presents lower and upper bounds on communication.
Section~\ref{cetm} discusses computing eigenvectors from matrices in Schur form.
Section~\ref{sec_SBR} discusses successive band reduction.
Finally, Section \ref{Conc} draws conclusions and presents some open problems.

\section{Randomized Rank-Revealing Decompositions} \label{rrdr}

Let $A$ be an $n \times n$ matrix with singular values $\sigma_1 \geq \sigma_2 \geq \ldots \geq \sigma_n$, and assume that there is a ``gap'' in the singular values at level $k$, that is, $\sigma_1/\sigma_k = O(1)$, while $\sigma_k/\sigma_{k+1} \gg 1$. 

Informally speaking, a decomposition of the form $A = URV$ is called \emph{rank revealing} if the following conditions are fulfilled:
\begin{itemize}
\item[1)] $U$ and $V$ are orthogonal/unitary and $R = \left [ \begin{array}{cc} R_{11} & R_{12} \\ O & R_{22} \end{array} \right ]$ is upper triangular, with $R_{11}$  $k \times k$ and  $R_{22}$  $(n-k) \times (n-k)$;
\item[2)] $\sigma_{min}(R_{11})$ is a ``good'' approximation to $\sigma_k$ (at most a factor of a low-degree polynomial in $n$ away from it),
\item[(3)] $\sigma_{max}(R_{22})$ is a ``good'' approximation to $\sigma_{k+1}$ (at most a factor of a low-degree polynomial in $n$ away from it);
\item[(4)] In addition, if $||R_{11}^{-1}R_{12}||_2$ is small (at most a low-degree polynomial in $n$), then the rank-revealing factorization is called  \emph{strong} (as per \cite{GE96}).

\end{itemize}

Rank revealing decompositions are used in rank determination \cite{stewart84}, least square computations \cite{CH92}, 
condition estimation \cite{bischof90a}, etc.,
as well as in divide-and-conquer algorithms for eigenproblems.
For a good survey paper, we recommend \cite{GE96}.

In the paper \cite{DDH07}, we have proposed a \emph{randomized} rank revealing factorization algorithm \textbf{RURV}. Given a matrix $A$, the routine computes a decomposition $A = URV$ with the property that $R$ is a rank-revealing matrix; the way it does it is by ``scrambling'' the columns of $A$ via right multiplication by a uniformly random orthogonal (or unitary) matrix $V^{H}$. The way to obtain such a random matrix (whose described distribution over the manifold of unitary/orthogonal matrices is known as \emph{Haar}) is to start from a matrix $B$ of independent, identically distributed normal variables of mean $0$ and variance $1$, denoted here and throughout the paper by $N(0,1)$. The orthogonal/unitary matrix $V$ obtained from performing the \textbf{QR} algorithm on the matrix $B$ is Haar distributed. 

Performing \textbf{QR} on the resulting matrix $A V^{H}=:\hat{A} = UR$ yields two matrices, $U$ (orthogonal or unitary) and $R$ (upper triangular), and it is immediate to check that $A = URV$. 

\begin{algorithm}
\protect\caption{Function $[U, R, V] =$\textbf{RURV}$(A)$, computes a randomized rank revealing decomposition $A = URV$, with $V$ a Haar matrix.} 
\begin{algorithmic}[1]
\label{rurv}
\STATE Generate a random matrix $B$ with i.i.d. $N(0,1)$ entries. 
\STATE $[V, \hat{R}] = $\textbf{QR}$(B)$.
\STATE $\hat{A} = A \cdot V^{H}$.
\STATE $[U,R] =$ \textbf{QR}$(\hat{A})$.
\STATE Output $R$.
\end{algorithmic}
\end{algorithm}

It was proven in \cite{DDH07} that, with high probability, \textbf{RURV} computes a good rank revealing decomposition of $A$ in the case of $A$ real. Specifically, the quality of the rank-revealing decomposition depends on computing the asymptotics of $f_{r,n}$, the smallest singular value of an $r \times r$ submatrix of a Haar-distributed orthogonal $n \times n$ matrix. All the results of \cite{DDH07} can be extended verbatim to Haar-distributed unitary matrices; however, the analysis employed in \cite{DDH07} is not optimal. In \cite{dumitriu10a}, the asymptotically exact scaling and distribution of $f_{r,n}$ were computed for all regimes of growth $(r,n)$ (naturally, $0 < r < n$). Essentially, the bounds state that $\sqrt{r(n-r)} f_{r,n}$ converges in law to a given distribution, both for the real and for the complex cases. 

The result of \cite{DDH07} states that \textbf{RURV} is, with high probability, a rank-revealing factorization. Here we strengthen these results to argue that it if in fact a \emph{strong} rank-revealing factorization, in the sense of \cite{GE96}, since with high probability $||R_{11}^{-1} R_{12}||$ will be small. We obtain the following theorem. 

\begin{theorem} \label{thm_rurv}  Let $A$ be an $n \times n$ matrix with singular values $\sigma_1, \ldots, \sigma_r, \sigma_{r+1}, \ldots, \sigma_n$. Let $\epsilon>0$ be a small number. Let $R$ be the matrix produced by the \textbf{RURV} algorithm on $A$.  Assume that $n$ is large and that $\frac{\sigma_1}{\sigma_r} = M$ and that $\frac{\sigma_{r}}{\sigma_{r+1}} > C n$ for some constants $M$ and $C>1$, independent of $n$. 

There exist constants $c_1$, $c_2$, and $c_3$ such that, with probability $1 - O(\epsilon)$, the following three events occur:
\begin{eqnarray*}
c_1 \frac{\sigma_{r}}{\sqrt{r (n-r)}} & \leq & \sigma_{\min}(R_{11}) \leq \sqrt{2} \sigma_r ~,\\
\sigma_{r+1} & \leq & \sigma_{\max} (R_{22}) \leq \left (3 M^3 \frac{C^2}{C^2-1}\right) \frac{r^2(n-r)^2 \sigma_{r+1}}{c_2}~, ~~~\mbox{and}\\
||R_{11}^{-1} R_{12}||_2 & \leq & c_3 \sqrt{r (n-r)}~.
\end{eqnarray*}
\end{theorem}

\begin{proof} There are two cases of the problem, $r \leq n/2$ and $r> n/2$. 
Let $V$ be the Haar matrix used by the algorithm. 
From Theorem 2.4-1 in \cite{golubvanloan}, 
the singular values of $V(1;r, 1:r)$ when $r>n/2$ consist 
of $(2r-n)$ $1$'s and the singular values of $V((r+1):n ,(r+1):n)$. 
Thus, the case $r>n/2$ reduces to the case $r \leq n/2$.

The first two relationships follow immediately from \cite{DDH07} and \cite{dumitriu10a}; the third we will prove here. 

We use the following notation. Let $A = P \Sigma Q^{H} = P \cdot \diag(\Sigma_1, \Sigma_2) \cdot Q^{H}$ be the singular value decomposition of $A$, where $\Sigma_1 = \diag(\sigma_1, \ldots, \sigma_r)$ and $\Sigma_2 = \diag(\sigma_{r+1}, \ldots, \sigma_n)$. Let $V^H$ be the random unitary matrix in \textbf{RURV}. Then $X = Q^{H}V^{H}$ has the same distribution as $V^{H}$, by virtue of the fact that $V$'s distribution is uniform over unitary matrices. 

Write 
\[
X = \left [ \begin{array}{cc} X_{11} & X_{12} \\ X_{21} & X_{22} \end{array} \right ]~,
\]
where $X_{11}$ is $r \times r$, and $X_{22}$ is $(n-r) \times (n-r)$. 

Then 
\[
U^H P \cdot \Sigma X = R~;
\]
denote $\Sigma \cdot X = [Y_1, Y_2]$ where $Y_1$ is an $n \times r$ matrix and $Y_2$ in an $(n-r) \times n$ one. Since $U^{H}P$ is unitary, in effect 
\[
R_{11}^{-1} R_{12} = Y_1^{+} Y_2~,
\]
where $Y_1^{+}$ is the pseudoinverse of $Y_1$, i.e. $Y_1^{+} = (Y_1^H Y_1)^{-1} Y_1^H$. We obtain that
\[
R_{11}^{-1} R_{12} = \left (X_{11}^{H} \Sigma^2 X_{11} + X_{21}^{H} \sigma_2^2 X_{21} \right)^{-1} \left ( X_{11}^{H} \Sigma_1^2 X_{12} + X_{21}^{H} \Sigma_2^2 X_{22} \right)~.
\]
Examine now
\[
T_1 := \left (X_{11}^{H} \Sigma_1^2 X_{11} + X_{21}^H \Sigma^2_2 X_{21} \right )^{-1} X_{11}^{H} \Sigma_1^2 X_{12} = X_{11}^{-1} \left ( \Sigma_1^2 + (X_{21} X_{11}^{-1})^{H} \Sigma_2^2 (X_{21} X_{11}^{-1}) \right )^{-1} \Sigma_1^2 X_{12}~.
\]
Since $X_{12}$ is a submatrix of a unitary matrix, $||X_{12}|| \leq 1$, and thus
\begin{eqnarray} \label{r-bound}
||T_1|| \leq ||X_{11}||^{-1}  || I_r + \Sigma_1^{-2} (X_{21} X_{11}^{-1})^{H} \Sigma_2^2 (X_{21} X_{11}^{-1}) ||^{-1} \leq \frac{1}{\sigma_{\min}(X_{11})} \cdot \frac{1}{1 - \sigma_{r+1}^2 \sigma_{\min}(X_{11})^2 /\sigma_r^2}~.
\end{eqnarray}
Given that $\sigma_{\min}(X_{11}) = O( (\sqrt{r (n-r)})^{-1})$ with high probability and $\sigma_r /\sigma_{r+1} \geq C n$, it follows that the denominator of the last fraction in \eqref{r-bound} is $O(1)$. Therefore, there must exist a constant such that, with probability bigger than $1 - \epsilon$, $||T_1|| \leq c_3 \sqrt{r(n-r)}$. Note that $c_3$ depends on $\epsilon$.

The same reasoning applies to the term
\[
T_2:= \left (X_{11}^{H} \Sigma_1^2 X_{11} + X_{21}^H \Sigma^2_2 X_{21} \right )^{-1} X_{21} \Sigma_2^2 X_{22}~;
\]
to yield that 
\begin{eqnarray} \label{r-bound2}
||T_2|| \leq ||X_{11}^{-1}||^2 ||I_r + \Sigma_1^{-2} (X_{21} X_{11}^{-1})^{H} \Sigma_2^2 (X_{21} X_{11}^{-1}) ||^{-1} 
||\Sigma_1^{-2} || \cdot || \Sigma_2^{2}|| ,
\end{eqnarray}
because $||X_{22}|| \leq 1$.

The conditions imposed in the hypothesis ensure that 
$||X_{11}^{-1}|| \cdot  ||\Sigma_1^{-1}|| \cdot ||\Sigma_2|| = O(1)$, and thus $||T_2|| = O(1)$.

From \eqref{r-bound} and \eqref{r-bound2}, the conclusion follows. 
\end{proof}

\vspace{.25cm}

The bounds on $\sigma_{\min} (R_{11})$ are as good as any deterministic algorithm would provide (see \cite{GE96}). However, the upper bound on $\sigma_{\max} (R_{22})$ is much weaker than in corresponding deterministic algorithms. We suspect that this may be due to the fact that the methods of analysis are too lax, and that it is not intrinsic to the algorithm. 

This theoretical sub-optimality will not affect the performance of the algorithm in practice, as it will be used to differentiate between singular values $\sigma_{r+1}$ that are very small (polynomially close to $0$) and singular values $\sigma_r$ that are close to $1$, and thus far enough apart. 

Similarly to \textbf{RURV} we define the routine \textbf{RULV}, which performs the same kind of computation (and obtains a rank revealing decomposition of $A$), but uses \textbf{QL} instead of \textbf{QR}, and thus obtains a lower triangular matrix in the middle, rather than an upper triangular one. 

Given \textbf{RURV} and \textbf{RULV}, we now can give a method to find a randomized rank-revealing factorization for a product of matrices and inverses of matrices, \emph{without actually computing any of the inverses or matrix products}. This is a very interesting and useful procedure in itself, but we will also use it in Sections \ref{Algs} and \ref{CRA} in the analysis of a single step of our Divide-and-Conquer algorithms. 

Suppose we wish to find a randomized rank-revealing factorization $M_k = URV$ for the matrix $M_k = A_1^{m_1} \cdot A_2^{m_2} \cdot \ldots A_k^{m_k}$, where $A_1, \ldots, A_k$ are given matrices, and $m_1, \ldots, m_k \in \{-1,1\}$, without actually computing $M_k$ or any of the inverses. 

Essentially, the method performs \textbf{RURV} or, depending on the power, \textbf{RULV}, on the last matrix of the product, and then uses a series of \textbf{QR}/\textbf{RQ} to ``propagate'' an orthogonal/unitary matrix to the front of the product, while computing factor matrices from which (if desired) the upper triangular $R$ matrix can be obtained. A similar idea was explored by G.W. Stewart in \cite{Stewart95} to perform graded \textbf{QR}; although it was suggested that such techniques can be also applied to algorithms like $\textbf{URV}$, no randomization was used. 

The algorithm is presented in pseudocode below. 

\begin{algorithm}
\protect\caption{Function $U =$\textbf{GRURV}$(k; A_1, \ldots, A_k; m_1, \ldots, m_k)$, computes the $U$ factor in a randomized rank revealing decomposition of the product matrix $M_k = A_1^{m_1} \cdot A_2^{m_2} \cdot \ldots A_k^{m_k}$, where $m_1, \ldots, m_k \in \{-1,1\}$.} 
\label{grurv}\begin{algorithmic}[1]
\IF{$m_k = 1$,}
\STATE $[U,R_k,V] = \textbf{RURV}(A_k)$
\ELSE
\STATE $[U,L_k,V] = \textbf{RULV}(A_k^{H})$
\STATE $R_k =L_k^{H}$ 
\ENDIF
\STATE $U_{current} = U$
\FOR{$i = k-1$ downto $1$}
\IF{$m_i = 1$,}
\STATE $[U,R_i] = \textbf{QR}(A_i \cdot U_{current})$
\STATE $U_{current} = U$
\ELSE 
\STATE $[U,R_i] = \textbf{RQ}(U_{current}^{H} \cdot A_i)$
\STATE $U_{current} = U^H$
\ENDIF
\ENDFOR
\RETURN $U_{current}$, optionally $V, R_1, \ldots, R_k$
\end{algorithmic}
\end{algorithm}

\begin{lemma}
\textbf{GRURV} (\emph{Generalized Randomized URV}) computes the rank-revealing decomposition $M_k = U_{current}R_1^{m_1} \ldots R_k^{m_k} V$. 
\end{lemma}

\begin{proof} Let us examine the case when $k = 2$ ($k>2$ results immediately through simple induction). 

Let us examine the cases:
\begin{enumerate}
\item $m_2 = 1$. In this case, $M_2 = A_1^{m_1} A_2$; the first \textbf{RURV} yields
$M_2 = A_1^{m_1} UR_2V$. \begin{enumerate}
\item if $m_1 = 1$, $M_2 = A_1 UR_2V$; performing \textbf{QR} on $A_1 U$ yields
$M_2= U_{current} R_1 R_2 V$.
\item if $m_1 = -1$, $M_2 = A_1^{-1} UR_2V$; performing \textbf{RQ} on $U^{H}A_1$ yields $M_2 = U_{current} R_1^{-1} R_2 V$.
\end{enumerate}
\item $m_2 = -1$. In this case, $M_2 = A_1^{m_1} A_2^{-1}$; the first \textbf{RULV} yields $M_2 = A_1^{m_1} U L_2^{-H} V = A_1^{m_1} U R_2^{-1} V$. \begin{enumerate}
\item if $m_1 = 1$, $M_2 = A_1 U L_2^{-H} V = A_1 U R_2^{-1} V$; performing \textbf{QR} on $A_1 U$ yields $M_2 = U_{current} R_1 R_2^{-1} V$.
\item finally, if $m_2 = -1$,  $M_2 = A_1^{-1} U L_2^{-H} V = A_1^{-1} U R_2^{-1}V$; performing \textbf{RQ} on $U^{H}A_1$ yields $M_2 = U_{current} R_1^{-1} R_2^{-1} V$.
\end{enumerate}
\end{enumerate}

Note now that in all cases $M_k = U_{current} R_1^{m_1} \ldots R_k^{m_k} V$. Since the inverse of an upper triangular matrix is upper triangular, and since the product of two upper triangular matrices is upper triangular, it follows that $R:=R_1^{m_1} \ldots R_k^{m_k}$ is upper triangular. Thus, we have obtained a rank-revealing decomposition of $M_k$; the same rank-revealing decomposition as if we have performed $QR$ on $M_k V^{H}$.
\end{proof}

By using the stable \textbf{QR} and \textbf{RQ} algorithms described in \cite{DDH07}, as well as \textbf{RURV}, we can guarantee the following result.

\begin{theorem} \label{thm_grurv} The result of the algorithm \textbf{GRURV} is the same as the result of 
\textbf{RURV} on the (explicitly formed) matrix $M_k$; therefore, given a large gap in the 
singular values of $M_k$, $\sigma_{r+1} \ll \sigma_r \sim \sigma_1$, 
the algorithm \textbf{GRURV} produces a rank-revealing decomposition with high probability.
\end{theorem}

Note that we may also return the matrices $R_1, \ldots, R_k$, from which the factor $R$ can later 
be reassembled, if desired (our algorithms only need $U_{current}$, not $R =R_1^{m_1} \ldots R_k^{m_k}$,
and so we will not compute it).


\section{Divide-and-Conquer:  Four Versions} \label{Algs} 

As mentioned in the Introduction, the idea for the divide-and-conquer algorithm we propose has been first introduced in \cite{malyshev89}; it was subsequently made stable in \cite{baidemmelgu94}, and has then been modified in \cite{DDH07} to include a randomized rank-revealing decomposition, in order to minimize the complexity of linear algebra (reducing it to the complexity of matrix multiplication, while keeping algorithms stable). 
The version of \textbf{RURV} presented in this paper is the \emph{only} complete one, 
through the introduction of \textbf{GRURV}; 
also, the new analysis shows that it has stronger rank-revealing properties than previously shown. 

Since in \cite{DDH07} we only sketched the merger of the rank-revealing decomposition and the eigenvalue/singular value decomposition algorithms, we present the full version in this section, as it will be needed to analyze communication. We give the most general version \textbf{RGNEP} (which solves the generalized, non-symmetric eigenvalue problem), as well as versions \textbf{RNEP} (for the non-symmetric eigenvalue problem), \textbf{RSEP} (for the symmetric eigenvalue problem), and \textbf{RSVD} (for the singular value problem). The acronyms here are the same used in LAPACK, with the exception of the first letter ``R'', which stands for ``randomized''.

In this section and throughout the rest of the paper, \textbf{QR}, \textbf{RQ}, \textbf{QL}, and \textbf{LQ} represent functions/routines computing the corresponding factorizations. Although these factorizations are well-known, the implementations of the routines are not the ``standard'' ones, but the communication-avoiding ones described in \cite{BDHS10}. We only consider the latter, for the purpose of minimizing communication.

\subsection{Basic splitting algorithms (one step)}

Let $\mathcal{D}$ be the interior of the unit disk. Let $A$, $B$ be two $n \times n$ matrices with the property that the pencil $(A,B)$ has no eigenvalues \emph{on} $\mathcal{D}$. The algorithm below ``splits'' the eigenvalues of the pencil $(A,B)$ into two sets; the ones \emph{inside} $\mathcal{D}$ and the ones \emph{outside} $\mathcal{D}$. 

The method is as follows:
\begin{enumerate}
\item Compute $(I+(A^{-1}B)^{2^k})^{-1}$ implicitly. 
Roughly speaking, this maps the eigenvalues inside the unit circle to $0$ and those outside to $1$.
\item Compute a rank-revealing decomposition to find the right deflating subspace corresponding 
to eigenvalues inside the unit circle. 
This is spanned by the leading columns of a unitary matrix $Q_R$. 
\item Analogously compute $Q_L$ from $(I+(A^{-H}B^H)^{2^k})^{-1}$.
\item Output the block-triangular matrices
\[
\hat{A} = Q_L^H A Q_R = \left ( \begin{array}{cc} A_{11} & A_{12} \\ E_{21} & A_{22} \end{array} \right )~,~~~~~\hat{B} = Q_L^H B Q_R = \left ( \begin{array}{cc} B_{11} & B_{12} \\ F_{21} & B_{22} \end{array} \right )~.
\]
\end{enumerate}

The pencil $(A_{11}, B_{11})$ has no eigenvalues inside $\mathcal{D}$; 
the pencil $(A_{22}, B_{22})$ has no eigenvalues outside $\mathcal{D}$. 
In exact arithmetic and after complete convergence, we should have 
$E_{21} = F_{21} = 0$. In the presence of floating point error with a finite number of iterations, we use $||E_{21}||_1/||A||_1$ and $||F_{21}||_1/||B||_1$ to measure the stability of the computation. 

To simplify the main codes, we first write a routine to perform (implicitly) repeated squaring of the quantity $A^{-1}B$. The proofs for correctness are in the next section.

\begin{algorithm}[ht!]
\protect\caption{Function \textbf{IRS}, performs implicit repeated squaring of the quantity $A^{-1}B$ for a pair of matrices $(A,B)$}
\begin{algorithmic}[1]
\label{rsq}
\STATE Let $A_0 = A$ and $B_0 = B$; j = 0.
\REPEAT 
\STATE \begin{eqnarray*} \left ( \begin{array}{c} B_j \\ - A_j \end{array} \right ) & = & \left ( \begin{array}{cc} Q_{11} & Q_{12} \\ Q_{21} & Q_{22} \end{array} \right ) \left ( \begin{array}{c} R_j \\ 0 \end{array} \right )~, \\
A_{j+1} & = & Q^H_{12} A_j~, \\
B_{j+1} & = &  Q^H_{22} B_j~, \end{eqnarray*}
\IF{$ ||R_j - R_{j-1}||_1 \leq  \tau ||R_{j-1}||_1~,$   \hspace{3cm} \emph{... convergence!}

}
\STATE $ p = j+1~,$ 
\ELSE  
\STATE $j = j+1$
\ENDIF
\UNTIL{convergence or $j>maxit$.}
\RETURN $A_p$, $B_p$.
\end{algorithmic}
\end{algorithm}

Armed with repeated squaring, we can now perform a step of divide-and-conquer. We start with the most general case of a two-matrix pencil.

\begin{algorithm}[ht!]
\protect\caption{Function \textbf{RGNEP}, performs a step of divide-and-conquer on a pair of matrices $A$ and $B$}
\begin{algorithmic}[1]
\label{rgnep}
\STATE $[A_p, B_p] = $\textbf{IRS}$(A,B)$.
\STATE $Q_R = $\textbf{GRURV}$(2, A_p+B_p, A_p, -1, 1)$.
\STATE $[A_p, B_p] = $\textbf{IRS}$(A^{H}, B^{H})$.
\STATE $Q_L = $\textbf{GRURV}$(2, A_P^{H}, (A_p+B_p)^{H}, 1, -1)$. 
\STATE \begin{eqnarray} \label{unu} \hat{A} &:=& Q_L^{H} A Q_R = \left ( \begin{array}{cc} A_{11} & A_{12} \\ E_{21} & A_{22} \end{array} \right )~,\\ 
\label{doi} \hat{B} &:=& Q_L^H BQ_R = \left ( \begin{array}{cc} B_{11} & B_{12} \\ F_{21} & B_{22} \end{array} \right )~;
\end{eqnarray} 
the dimensions of the subblocks are chosen so as to minimize $||E_{21}||_1/||A||_1+||F_{21}||_1/||B||_1$;
\RETURN $(\hat{A}, \hat{B}, Q_L, Q_R)$ and the dimensions of the subblocks.
\end{algorithmic}
\end{algorithm}

\vspace{.35cm}

The algorithm \textbf{RGNEP} starts with the right deflating subspace:
line 1 does implicit repeated squaring of $A^{-1}B$,
followed by a rank-revealing decomposition of $(A_p+B_p)^{-1}A_p$ (line 2).
The left deflating subspace is handled similarly in lines 3 and 4.
Finally, line 5 divides the pencil by choosing the split that minimizes the sum 
of the relative norms of the bottom-left submatrices, 
e.g. $||E_{21}||_1/||A||_1+||F_{21}||_1/||B||_1$. 
If this norm is small enough, convergence is successful.  In this case
$E_{21}$ and $F_{21}$ may be zeroed out, dividing the problem into
two smaller ones, given by $(A_{11},B_{11})$ and $(A_{22},B_{22})$.
Note that if more than one pair of blocks $(E_{21},F_{21})$ is small enough
to zero out, the problem may be divided into more than two smaller ones.

\begin{algorithm}[ht!]
\protect\caption{Function \textbf{RNEP}, performs a step of divide-and-conquer on a non-hermitian matrix $A$; $I$ here is the identity matrix.}
\begin{algorithmic}[1]
\label{rnep}
\STATE $[A_p, B_p] = $\textbf{IRS}$(A,I)$.
\STATE $Q = $\textbf{GRURV}$(2, A_p+B_p, A_p, -1, 1)$.
\STATE \begin{eqnarray*} \label{trei} \hat{A} &:=& Q^{H} A Q = \left ( \begin{array}{cc} A_{11} & A_{12} \\ E_{21} & A_{22} \end{array} \right )~,
\end{eqnarray*}
the dimensions of the subblocks are chosen so as to minimize $||E_{21}||_1/||A||_1$;
\RETURN $(\hat{A}, Q)$ and the dimensions of the blocks.
\end{algorithmic}
\end{algorithm}

\vspace{.35cm}

The algorithm \textbf{RNEP} deals with the non-symmetric eigenvalue problem, 
that is, when $A$ is non-hermitian and $B = I$. In this case, we skip the 
computation of the left deflating subspace, since it is the same as the right one; 
also, it is sufficient to consider $||E_{21}||_1 / ||A||_1$, since it is the 
backward error in the computation. 
Note that \eqref{doi} is not needed, as in exact arithmetic $\hat{B} = I$.

\begin{algorithm}[ht!]
\protect\caption{Function \textbf{RSEP}, performs a step of divide-and-conquer on a hermitian matrix $A$; $I$ here is the identity matrix.}
\begin{algorithmic}[1]
\label{rsep}
\STATE $[A_p, B_p] = $\textbf{IRS}$(A,I)$.
\STATE $[U, R_1, V] = $\textbf{GRURV}$(2, A_p+B_p, A_p, -1, 1)$.
\STATE \begin{eqnarray*} \label{five} \hat{A} &:=& Q^{H} A Q \\
\label{six} \hat{A} & := & \frac{\hat{A}+\hat{A}^{H}}{2} \\
\label{seven} \hat{A} & := & \left ( \begin{array}{cc} A_{11} & E_{21}^{T} \\ E_{21} & A_{22} \end{array} \right )~,
\end{eqnarray*}
the dimensions of the subblocks are chosen so as to minimize $||E_{21}||_1/||A||_1$;
\RETURN $(\hat{A}, Q)$ and the dimensions of the blocks.
\end{algorithmic}
\end{algorithm}

\vspace{.35cm}

The algorithm \textbf{RSEP} is virtually identical to \textbf{RNEP}, with the only difference in the second equation of line 3, where we enforce the symmetry of $\hat{A}$. We choose to write two separate pieces of code for \textbf{RSEP} and \textbf{RNEP}, as the analysis of theses two codes in Section \ref{CRA} differs.

We choose to explain the routine for computing singular values of a matrix $A$,  rather than present it in pseudocode. Instead of calculating the singular values of $A$ directly, we construct the hermitian matrix $B = \left [\begin{array}{cc} 0 & A \\ A^{H} & 0 \end{array} \right]$ and compute its eigenvalues (which are the singular values of $A$ and their negatives) and eigenvectors (which are concatenations of left and right singular vectors of $A$). Thus, computing singular values completely reduces to computing eigenvalues. 

\subsection{One step: correctness and reliability of implicit repeated squaring} \label{CRA} 

Assume for simplicity that all matrices involved are invertible, and let us examine the basic step of the algorithm \textbf{IRS}. It is easy to see that
\[
\left ( \begin{array}{cc} Q_{11}^{H} & Q_{21}^{H} \\ Q_{12}^{H} & Q_{22}^{H} \end{array} \right ) \left( \begin{array}{c} B_j \\ - A_j \end{array} \right ) = \left ( \begin{array}{c} Q_{11}^{H} B_j - Q_{21}^{H} A_j \\ Q_{12}^{H} B_j - Q_{22}^{H} A_j \end{array} \right ) = \left ( \begin{array}{c} R_j \\ 0 \end{array} \right)~,
\]
thus $Q_{12}^{H} B_j = Q_{22}^{H} A_j$, and then $B_j A_j^{-1} = Q_{12}^{-H} Q_{22}$; finally, $$A_{j+1}^{-1} B_{j+1} = A_{j}^{-1} Q_{12}^{-H} Q_{22}^{H}B_j = (A_j^{-1} B_j )^2~,$$ proving that the algorithm \textbf{IRS} repeatedly squares the eigenvalues, sending those outside the unit circle to $\infty$, and the ones inside to $0$.

As in the proof of correctness from \cite{baidemmelgu94}, we note that
\[
(A_p+B_p)^{-1} A_p = (I+A_p^{-1} B_p)^{-1} = (I+(A^{-1}B)^{2^p})^{-1}~,
\]
and that the latter matrix approaches $P_{R, |z|>1}$,  the projector onto the right deflating subspace corresponding to eigenvalues outside the unit circle. 


\subsection{High level strategy for divide-and-conquer, single non-symmetric matrix case}
\label{sec:HLS:Nonsymmetric}

For the rest of this section, we will assume that the matrix $B = I$ and that the matrix $A$ is diagonalizable, with eigenvector matrix $S$ and spectral radius $\rho(A)$.

The first step will be to find (e.g., using Gershgorin bounds) a radius $R \geq \rho(A)$, so that we are assured that all eigenvalues are in the disk of radius $R$. Choose now a random angle $\theta$, uniformly from $[0, 2\pi]$, and draw a radial line $\mathcal{R}$ through the origin, at angle $\theta$. The line $\mathcal{R}$ intersects the circle centered at the origin of radius $R/2$ at two points defining a diameter of the circle, $A$ and $B$. We choose a random point $a$ uniformly on the segment $AB$ and construct a line $\mathcal{L}$ perpendicular to $AB$ through $a$.

The idea behind this randomization procedure is to avoid not just ``hitting'' the eigenvalues, but also ``hitting'' some appropriate $\epsilon$-pseudospectrum of the matrix, defined below (the right choice for $\epsilon$ will be discussed later).

\begin{definition} (Following \cite{TE05}.) The $\epsilon$-pseudospectrum of a matrix $A$ is defined as
\[
\Lambda_{\epsilon}(A) = \{z \in \mathbb{C} ~|~ z ~\mbox{is an eigenvalue of}~A+E~\mbox{for some}~E~\mbox{with}~||E||_2\leq \epsilon \}
\]
\end{definition}

We measure progress in two ways: one is when a split occurs, i.e., we separate some part of the pseudospectrum, and the second one is when a large piece of the disk is determined not to contain any part of the pseudospectrum. This explains limiting the choice of $\mathcal{L}$ to only half the diameter of the circle; this way, we ensure that as long as we are not ``hitting'' the $\epsilon$-pseudospectrum, we are making progress. See Figure \ref{pr_nopr}.

\begin{figure}[ht!] 
\begin{center}
\includegraphics[width=6in]{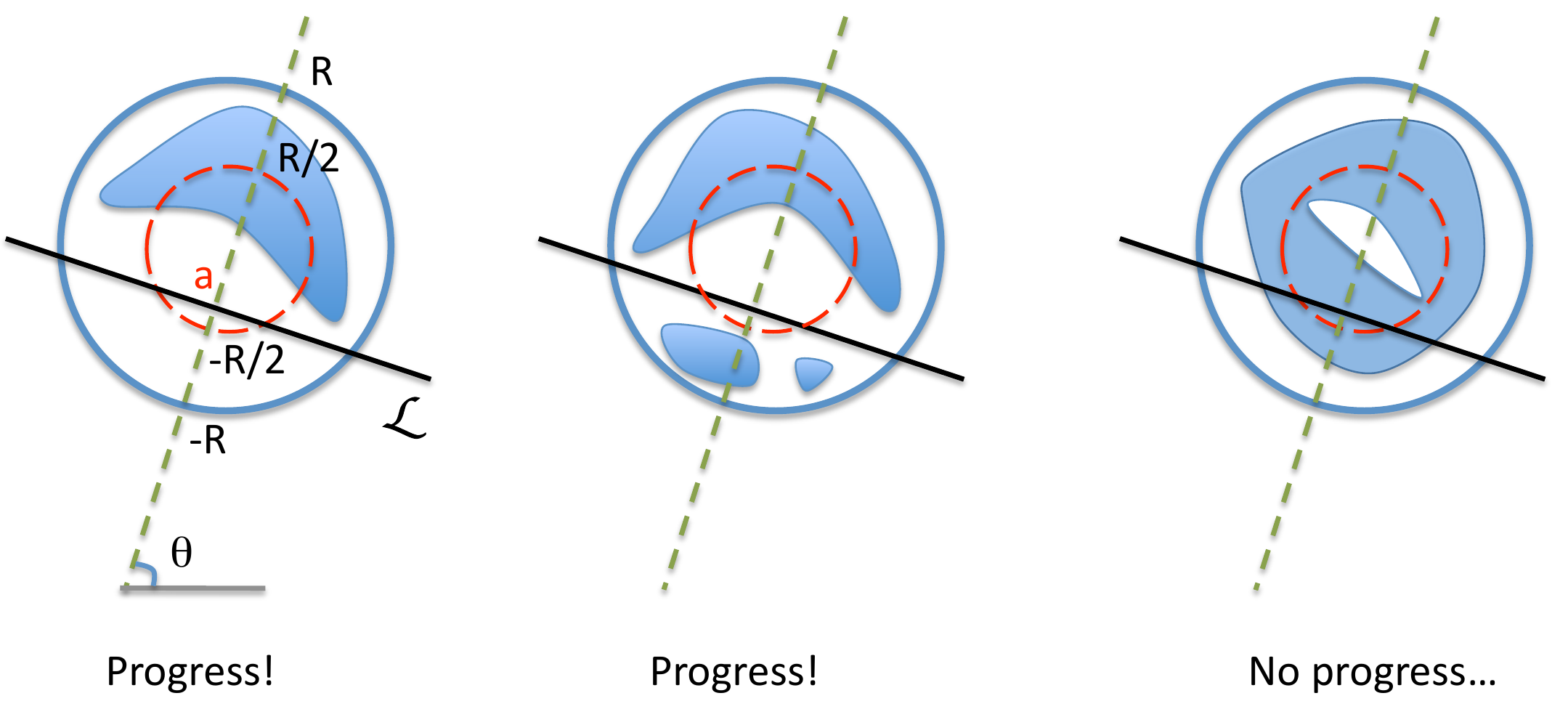}
\caption{The possible outcomes of choosing a splitting line. The shaded shapes represent pseudospectral components, the dashed inner circle has radius $R/2$, the dashed line is at the random angle $\theta$, and the solid line is $\mathcal{L}$.} 
\label{pr_nopr}
\end{center}
\end{figure}

\begin{remark} It is thus possible that the ultimate answer could be a convex hull of the pseudospectrum; this is an alternative to the output that classical (or currently implemented) algorithms yield, which samples the pseudospectrum. Since in general pseudospectra can be arbitrarily complicated (see \cite{trefethenreichel}), the amount of information gained may be limited; for example, the genus of the pseudospectrum will not be specified (i.e., the output will ignore the presence of ``holes'', as in the third picture on Figure \ref{pr_nopr}). However, in the case of convex or almost convex pseudospectra, a convex hull may be more informative than sampling.
\end{remark}
\begin{remark}
One could also imagine a method by which one could attempt to detect ``holes'' in the pseudospectrum, e.g., by sampling random points inside of the convex hull obtained and ``fitting'' disks centered at the points. Providing a detailed explanation for this case falls outside the scope of this paper. 
\end{remark}

Let us now compute a bound on the probability that the line $\mathcal{L}$ does not intersect a particular pseudospectrum of $A$ (which will depend on the machine precision, norm of $A$, etc.). To this extent, we will first examine the probability that the line $\mathcal{L}$ is at (geometrical) distance at least $\hat{\epsilon}$ from the \emph{spectrum} of $A$. 

Note that the problem is rotationally invariant; therefore it suffices to consider the case when $\theta = 0$, and simply examine the case when $\mathcal{L}$ is $x = a$, i.e., our random line is a vertical line passing through the point $a$ on the segment $[-R/2, R/2]$. 

Consider now the $n$ $\hat{\epsilon}$-balls that $\mathcal{L}$ must not intersect, and put a vertical ``tube'' around each one of them (as in Figure \ref{tube}). These tubes will intersect the segment $[-R,R]$ in $n$ (possibly overlapping) small segments (of length $2 \hat{\epsilon}$). In the worst-case scenario, they are all inside the segment $[-R/2, R/2]$, with no tubes overlapping. If the point $a$, chosen uniformly, falls into any one of them, the random line $\mathcal{L}$ will fall within $\hat{\epsilon}$ of the corresponding eigenvalue. This happens with probability at most $2 \hat{\epsilon} n/R$. 

The above argument proves the following lemma:

\begin{lemma} \label{dist}
Let $d_g$ be the geometrical distance between the (randomly chosen) line $\mathcal{L}$ and the spectrum. Then
\begin{eqnarray} \label{prob_bound}
P[d_g \geq \hat{\epsilon}] \geq \max \{1 - \frac{2n \hat{\epsilon}}{R}~, 0\}~.
\end{eqnarray}
\end{lemma}

\vspace{-.5cm}

\begin{figure}[ht!]
\begin{center}
\includegraphics[width=5in]{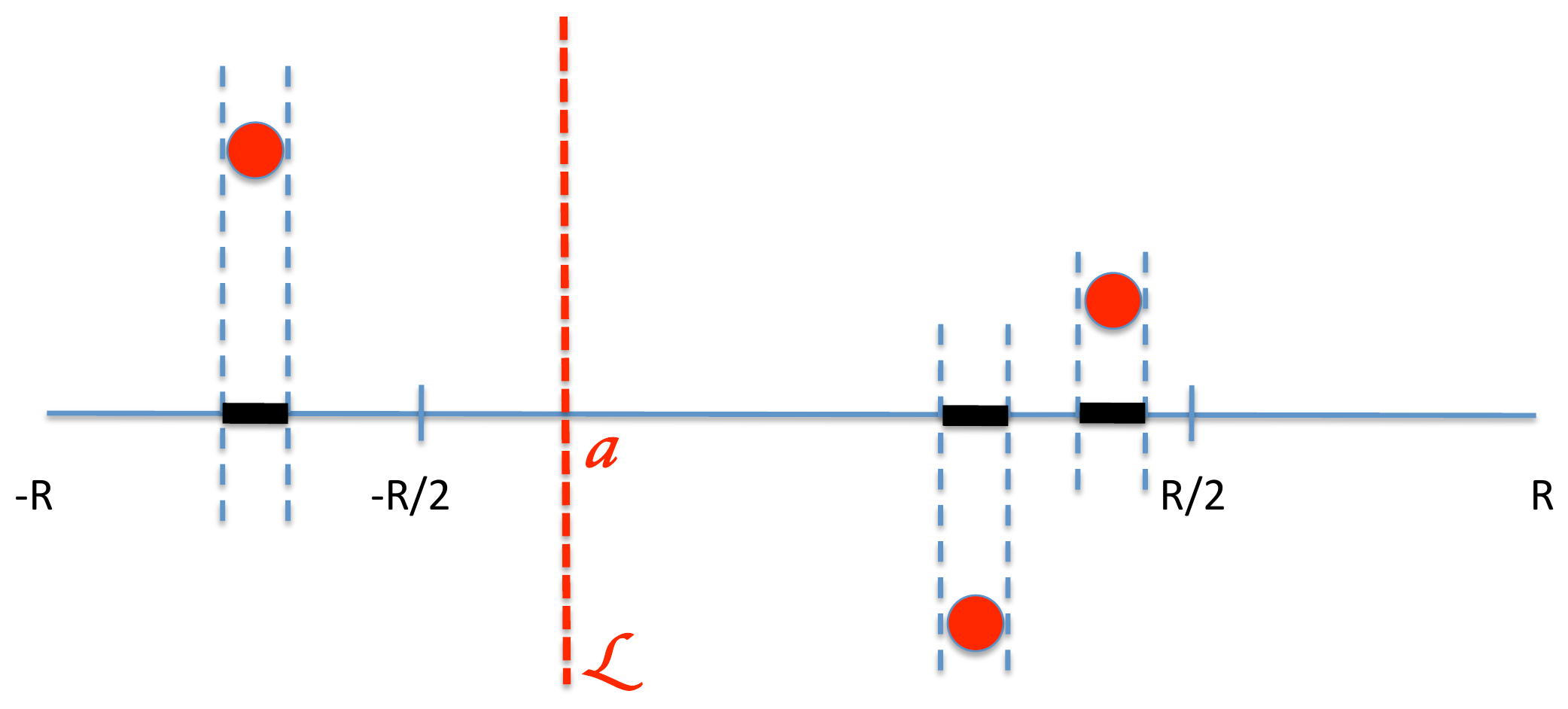}

\caption{Projecting the $\hat{\epsilon}$ tubes onto $[-R, R]$; the randomly chosen line $\mathcal{L}$ (i.e., $x=a$) does not intersect the (thick) projection segments (each of length $2 \hat{\epsilon}$) with probability at least $1 - \frac{2n \hat{\epsilon}}{R}$. The disks are centered at eigenvalues.} \label{tube}
\end{center}
\end{figure}


We would like now to estimate the number of iterations that the algorithm will need to converge, so we will use the results from \cite{baidemmelgu94}. To use them, we need to map the problem so that the splitting line becomes the unit circle. We will thus revisit the parameters that were essential in understanding the number of iterations in that case, and we will see how they change in the new context.

We start by recalling the notation $d_{(A,B)}$ for the ``distance to the nearest ill-posed problem'' (in our case, ``ill-posed'' means that there is an eigenvalue on the unit circle) for the matrix pencil $(A, B)$, defined as follows:
\[
d_{(A,B)} \equiv \inf\{||E||+||F||~: (A+E) - z(B+F) ~\mbox{is singular for some $z$ where $|z|=1$}\}~;
\]
according to Lemma 1, Section 5 from \cite{baidemmelgu94}, $d_{(A,B)} = \min_{\theta} \sigma_{\min} (A - e^{i\theta}B)$. 

We use now the estimate for the number $p$ of iterations given in Theorem 1 of Section 4 from \cite{baidemmelgu94}. Provided that the number of iterations, $p$, is larger than 
\begin{eqnarray} \label{p-eq}
p \geq \log_2 \frac{||(A, B)|| - d_{(A,B)}}{d_{(A,B)}}~,
\end{eqnarray}
the relative distance between $(A_p+B_p)^{-1} A_p$ and the projector $P_{R, |z|>1}$ may be bounded from above, as follows:
\[
err_p := \frac{||(A_p+B_p)^{-1} A_p - P_{R, |z|>1}||}{||P_{R, |z|>1}||} \leq \frac{2^{p+3} \left ( 1 - \frac{d_{(A,B)}}{||(A,B)||} \right )^{2^p}}{\max \left ( 0, 1 - 2^{p+2} \left (1 - \frac{d_{(A,B)}}{||(A,B)||} \right )^{2^p} \right )}~.
\]
This implies that (for $p$ satisfying \eqref{p-eq}) convergence is quadratic and the rate depends on $d_{(A,B)}/||(A,B)||$, the size of the smallest perturbation that puts an eigenvalue of $(A, B)$ on the unit circle. 

The new context, however, is that of lines instead of circles; so, if we start with $(A,I)$ and we choose a random line (w.l.o.g. assume that the line is $Re(z) = a$) that we then map to the unit circle using the map $f(z) = \frac{z-a+1}{z-a-1}$, the pencil $(A, I)$ gets mapped to $(A-(a-1)I, A-(a+1)I)$, and the distance to the closest ill-defined problem for the latter pencil is $d_{(A,I)} = \min_{\theta} \sigma_{\min} ((A- (a-1)I) - e^{i \theta} (A-(a+1)I))$.

Let now $$d_{\mathcal{L}}(A,B) \equiv \inf\{||E||+||F||~: (A+E) - z(B+F) ~\mbox{is singular for some $z$ where $Re(z) = a$}\}$$ be the size of the smallest perturbation that puts an eigenvalue of $(A,B)$ on the line $\mathcal{L}$.

It is a quick computation to show that 
\begin{eqnarray*}
d_{\mathcal{L}}(A,B) & = & \min_{\theta} 2  \mid \sin(\frac{\theta}{2}) \mid \sigma_{\min} (A-aB + \frac{e^{i \theta}+1}{1-e^{i\theta}}B)
\end{eqnarray*}
and from here, redefining $\theta$ as $\theta/2$, we obtain that
\begin{eqnarray} \label{d1}
d_{\mathcal{L}}(A,B) = 2 \min_{\theta \in [0, \pi]} \sigma_{\min} \Big ( \sin(\theta) (A-aB) + i \cos(\theta)B \Big)~.
\end{eqnarray}

Let us now define $n_a(A,B) = ||(A-aB+B, A-aB-B)||$.

Combining equations \eqref{p-eq} and \eqref{d1} for the case of a pencil $(A, I)$, we obtain that, if $p$, the number of iterations, is larger than $\log_2(n_a(A,I) /d_{\mathcal{L}}(A,I)-1)$, then the relative distance between the computed projector and the actual one decreases quadratically. This means that, if we want a relative error $err_p \leq \epsilon$, we will need $\log_2(n_a(A,I) /d_{\mathcal{L}}(A,I)-1) + \log_2 \log_2 \left ( \frac{1}{\epsilon} \right)$ iterations. 

How do $d_{\mathcal{L}}$ and $d_g$ relate? In general, the dependence is complicated (see \cite{trefethenreichel}), but a simple relationship is given by the following result by Bauer and Fike, which can be found for example as Theorem 2.3 from \cite{TE05}: $d_{\mathcal{L}}(A, I) \geq \frac{d_g(A, I)}{\kappa(S)}$, where we recall that $S$ is the eigenvector matrix. We obtain the following lemma.

\begin{lemma} \label{d2}
Provided that $\hat{\epsilon}$ is suitably small, then with probability at least $1 - \frac{2 n \hat{\epsilon}}{R}$, the number of iterations to make $err_p \leq \epsilon$ is at most
$$ p = \log_2(n_a(A,I) \kappa(S) / \hat{\epsilon} - 1) + \log_2 \log_2 \left ( \frac{1}{\epsilon} \right )~.$$
\end{lemma}

Finally, we have to consider $\hat{\epsilon}$. In practice, every non-trivial stable dense matrix algorithm has a backward error bounded by $\epsilon_m f(n) ||A||$, where $f(n)$ is a polynomial of small degree and $\epsilon_m$ is the machine precision. This is the minimum sort of distance that we can expect to be from the pseudospectrum. Thus, we must consider $\hat{\epsilon} = c \cdot \epsilon_{m} \cdot f(n) \cdot \kappa(S) ||A||$.

\begin{remark} Naturally, this makes sense only if $\hat{\epsilon}< 1$. \end{remark}

Recall that here $R$ is the bound on the spectral radius (perhaps obtained via Gershgorin bounds). In practice $R = O(||A||)$; assuming that we have prescaled $A$ so that $||A|| = O(1)$, it follows that $a$, as well as $n_a(A,I)$ are of the same order. The only question may be, what happens if we ``make progress'' without actually splitting eigenvalues, e.g., what if $||A||$ is much smaller than $R$ at first? The answer is that we decrease $R$ quickly, and here is a 2-step way to see it.

Suppose that we have localized all eigenvalues in the circular segment made by a chord $\mathcal{C}$ at distance at most $R/2$ from the center (see Figure \ref{2step} below). We will consider the worst case, that is, when we split no eigenvalues, but ``chop off'' the smallest area. We will proceed as follows: \begin{enumerate}
\item The next random line $\mathcal{L}$ will be picked to be perpendicular on $\mathcal{C}$, at distance at most $R/2$ from the center. All bounds on probabilities and number of iterations computed above will still apply.
\item If, at the next step, we once again split off a large part of the circle with no eigenvalues, the remaining area fits into a circle of radius at most $\sin(5 \pi/12)R$; if we wish to simplify the divide-and-conquer procedure, we can consider the slightly larger circle, recenter the complex plane, and continue. Note that at the most, the new radius is only a fraction of $\sin(5\pi/12) \approx .965$ of the old one.
\item If, at the second step, we split the matrix, we can simply examine the Gershgorin bounds for the new matrices and continue the recursion. 
\end{enumerate}

\begin{figure}[ht!]
\begin{center}
\includegraphics[height=4in]{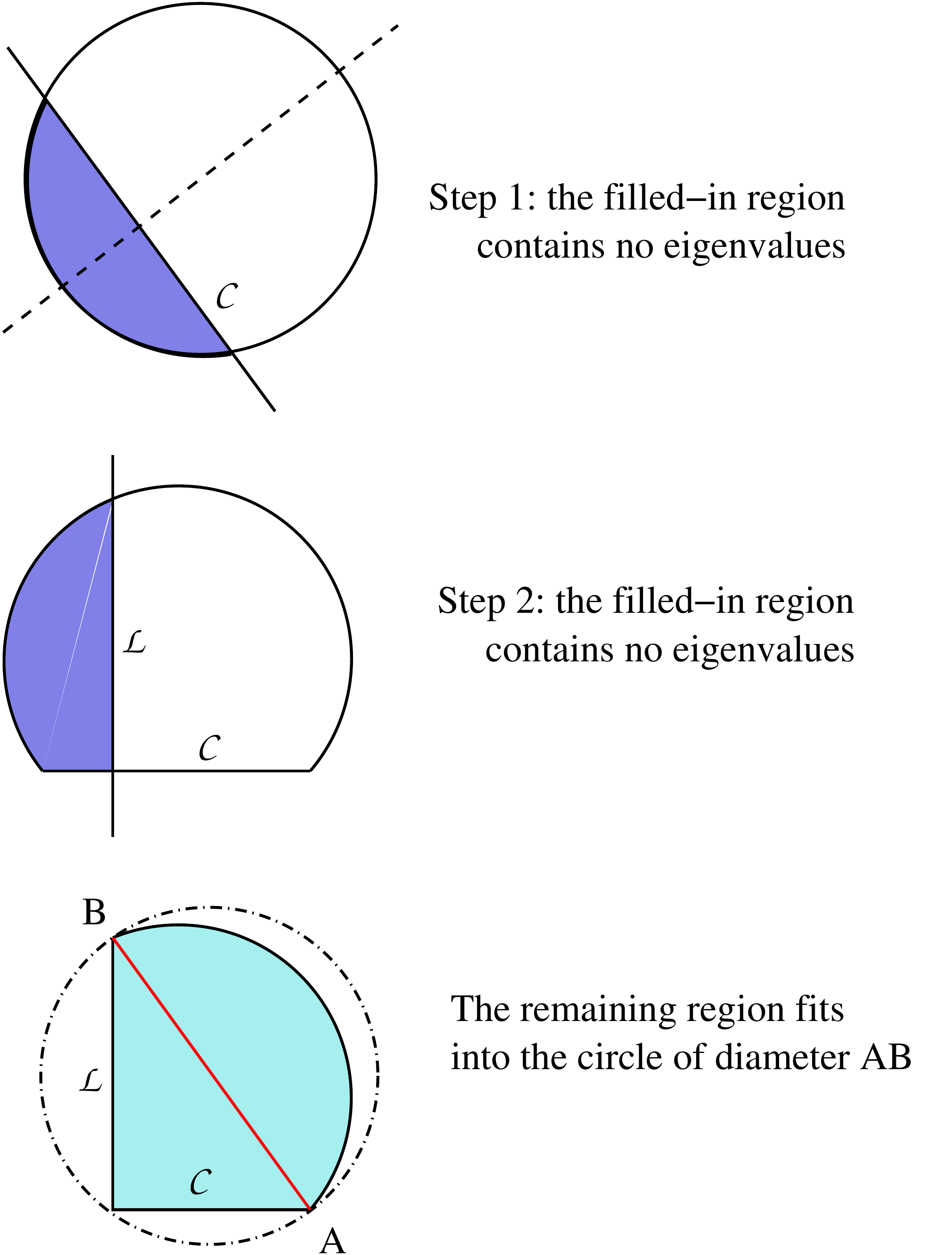}
\caption{The 2-step strategy reduces the radius of the bounding box.} \label{2step}
\end{center}
\end{figure}

Thus, we are justified in assuming that $R = O(||A||)$; to sum up, we have the following corollary:

\begin{corollary} \label{d3}
With probability at least $1 - c \cdot \epsilon \cdot n \cdot f(n) \cdot \kappa(S)$, the number of iterations needed for progress at every step is at most $ \log_2 \frac{1}{c \epsilon f(n)} + \log_2 \log_2 \frac{1}{\epsilon}$.
\end{corollary}

\subsubsection{Numerical experiments}
\label{sec:NumExp}

Numerical experiments for Algorithm~\ref{rnep} are given in Figures~\ref{fig:numexp_rand} and \ref{fig:numexp_hard}.  We chose four representative examples and present for each the distribution of the eigenvalues in the complex plane and the convergence plot of the backward error after a given number of iterations of implicit repeated squaring.  This backward error is $\frac{\|E_{21}\|_2}{\|A\|_2}$ in the notation given in Algorithm~\ref{rnep}, where the dimension of $E_{21}$ is chosen to minimize this quantity.  In practice, a different test of convergence can be used for the implicit repeated squaring, but for demonstrative purposes, we computed the invariant subspace after each iteration of repeated squaring in order to show the convergence of the backward error of an entire step of divide-and-conquer.

The examples presented in Figure~\ref{fig:numexp_rand} are normal matrices generated by choosing random eigenvalues (where the real and imaginary parts are each chosen uniformly from $[-1.5, 1.5]$) and applying a random orthogonal similarity transformation. The only difference is the matrix size.  We point out that the convergence rate depends not on the size of the matrix but rather the distance between the splitting line (the imaginary axis) and the closest eigenvalue (as explained in the previous section).  This geometric distance is the same as the distance to the nearest ill-posed problem because in this case, the condition number of the (orthogonal) diagonalizing similarity matrix is $1$.  The distance between the imaginary axis and the nearest eigenvalue is $8.36 \cdot 10^{-3}$ in Figure~\ref{subfig:numexp_rand100} ($n=100$) and $1.04 \cdot 10^{-3}$ in Figure~\ref{subfig:numexp_rand1000} ($n=1000$); note that convergence occurs at about the same iteration.

The examples presented in Figure~\ref{fig:numexp_hard} are meant to test the boundaries of convergence of Algorithm~\ref{rnep}.  While the eigenvalues in these examples are artificially chosen to be close to the splitting line (imaginary axis), our randomized higher level strategy for choosing splitting lines will ensure that such cases will be highly unlikely.  Figure~\ref{subfig:numexp_imagsplit25} presents an example where the eigenvalues are chosen with random imaginary part (between $-1.5$ and $1.5$) and real part set to $\pm 10^{-10}$ and the eigenvectors form a random orthogonal matrix. Because all the eigenvalues lie so close to the splitting line, convergence is much slower than in the examples with random eigenvalues.  Note that because the imaginary axis does not intersect the $\epsilon$-pseudospectrum (where $\epsilon$ is machine precision), convergence is still achieved within $40$ iterations.

The example presented in Figure~\ref{subfig:numexp_jordan32} is a 32-by-32 non-normal matrix which has half of its eigenvalues chosen randomly with negative real part and the other half set to $0.1$, forming a single Jordan block.  In this example, the imaginary axis intersects the $\epsilon$-pseudospectrum and so Algorithm~\ref{rnep} does not converge. Note that the eigenvalue distribution in the plot on the left shows $16$ eigenvalues forming a circle of radius $0.1$ centered at $0.1$.  The eigenvalue plots are the result of using standard algorithms which return $16$ distinct eigenvalues within the pseudospectral component containing $0.1$.

\begin{figure}[ht!]
\centering
\subfigure[$100\times100$ matrix with random eigenvalues]{
\includegraphics[scale=.7]{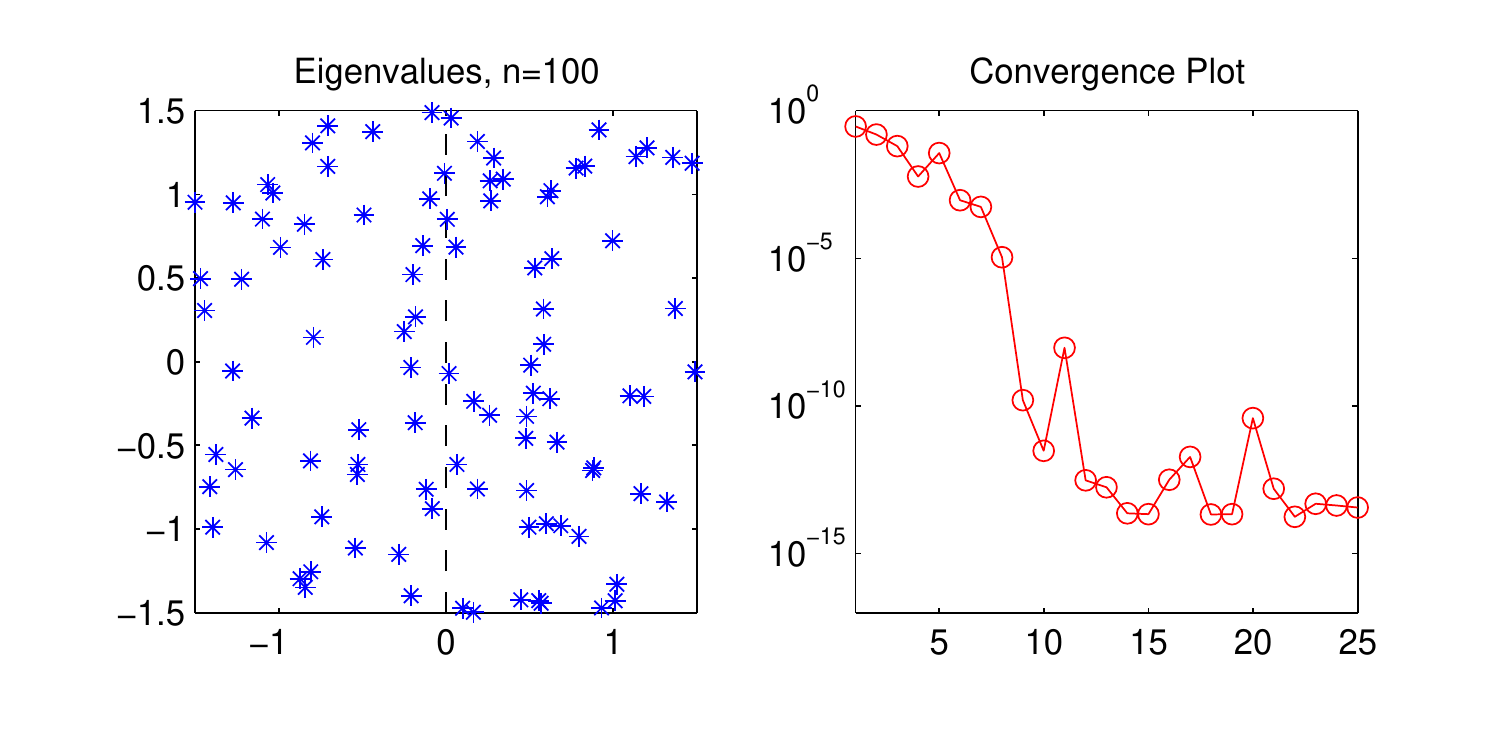}
\label{subfig:numexp_rand100}
}
\subfigure[$1000\times1000$ matrix with random eigenvalues]{
\includegraphics[scale=.7]{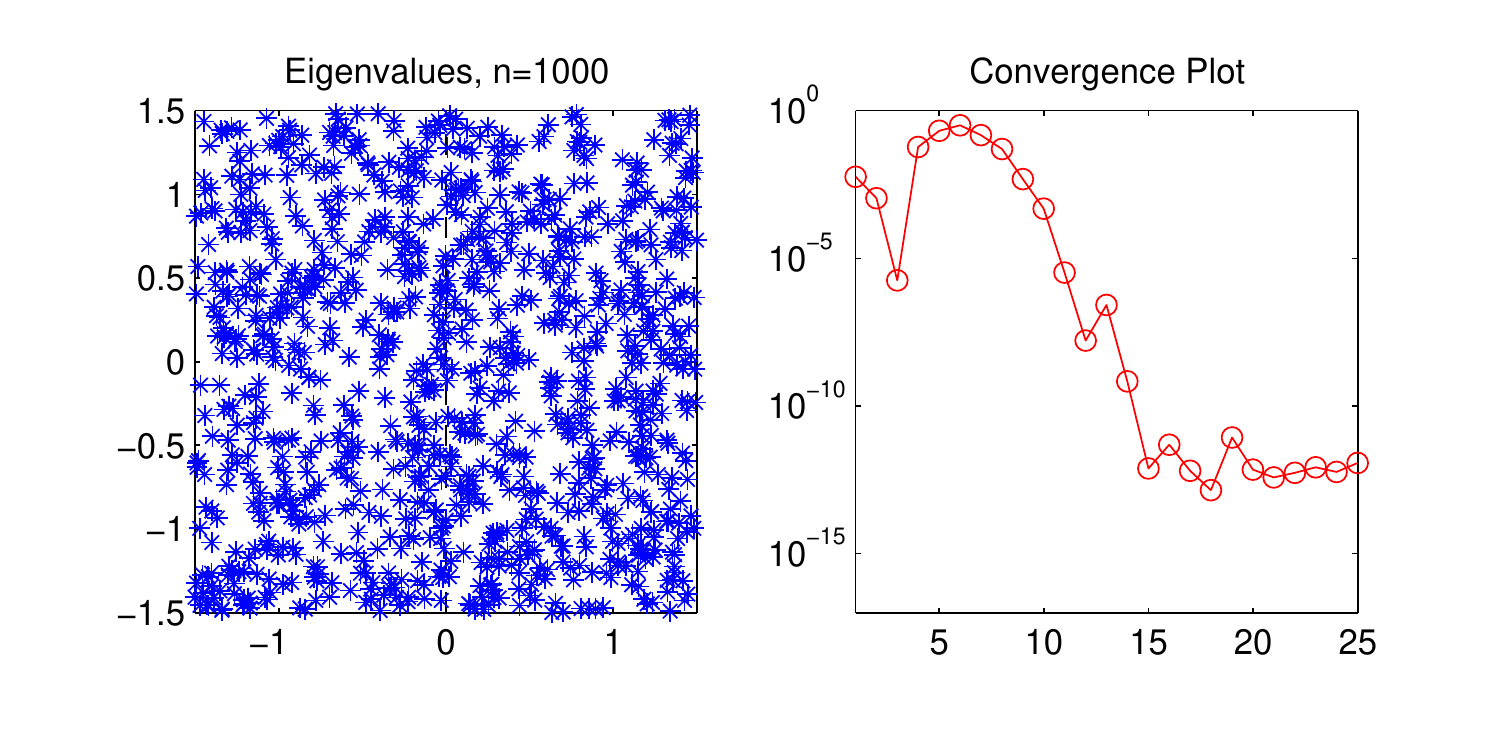}
\label{subfig:numexp_rand1000}
}
\caption{Numerical experiments using normal matrices for Algorithm~\textbf{RNEP}.  On the left of each subfigure
is a plot of the eigenvalues in the complex plane (as computed by standard algorithms)
with the imaginary axis (where the split occurs) highlighted.  In each case, the
eigenvalues have real and imaginary parts chosen uniformly from $[-1.5,1.5]$.  On the
right is a convergence plot of the backward error as described in
Section~\ref{sec:NumExp}.}
\label{fig:numexp_rand}
\end{figure}

\begin{figure}[ht!]
\centering
\subfigure[$25\times25$  matrix with eigenvalues whose real parts are $\pm 10^{-10}$]{
\includegraphics[scale=.7]{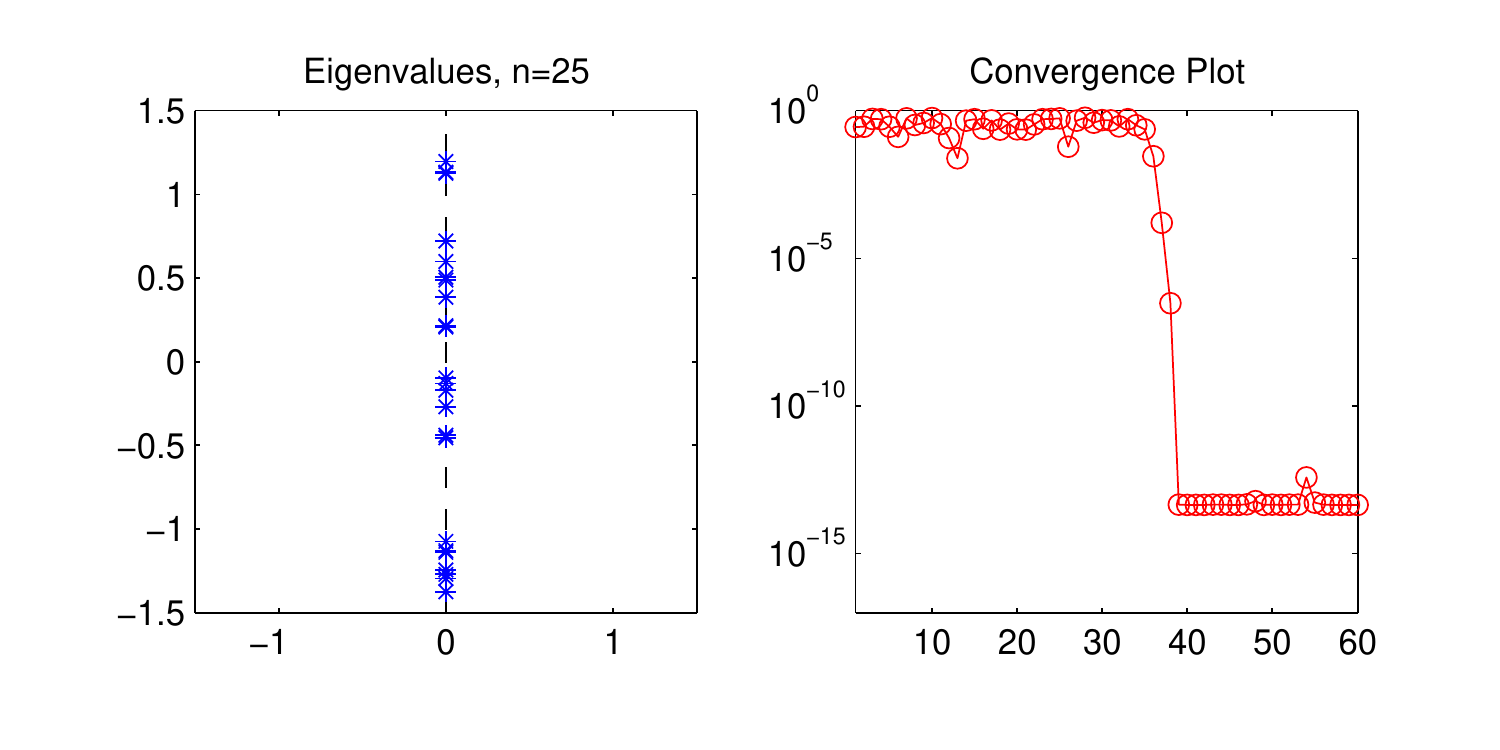}
\label{subfig:numexp_imagsplit25}
}
\subfigure[$32\times32$ matrix with half the eigenvalues forming Jordan block at
$0.1$]{
\includegraphics[scale=.7]{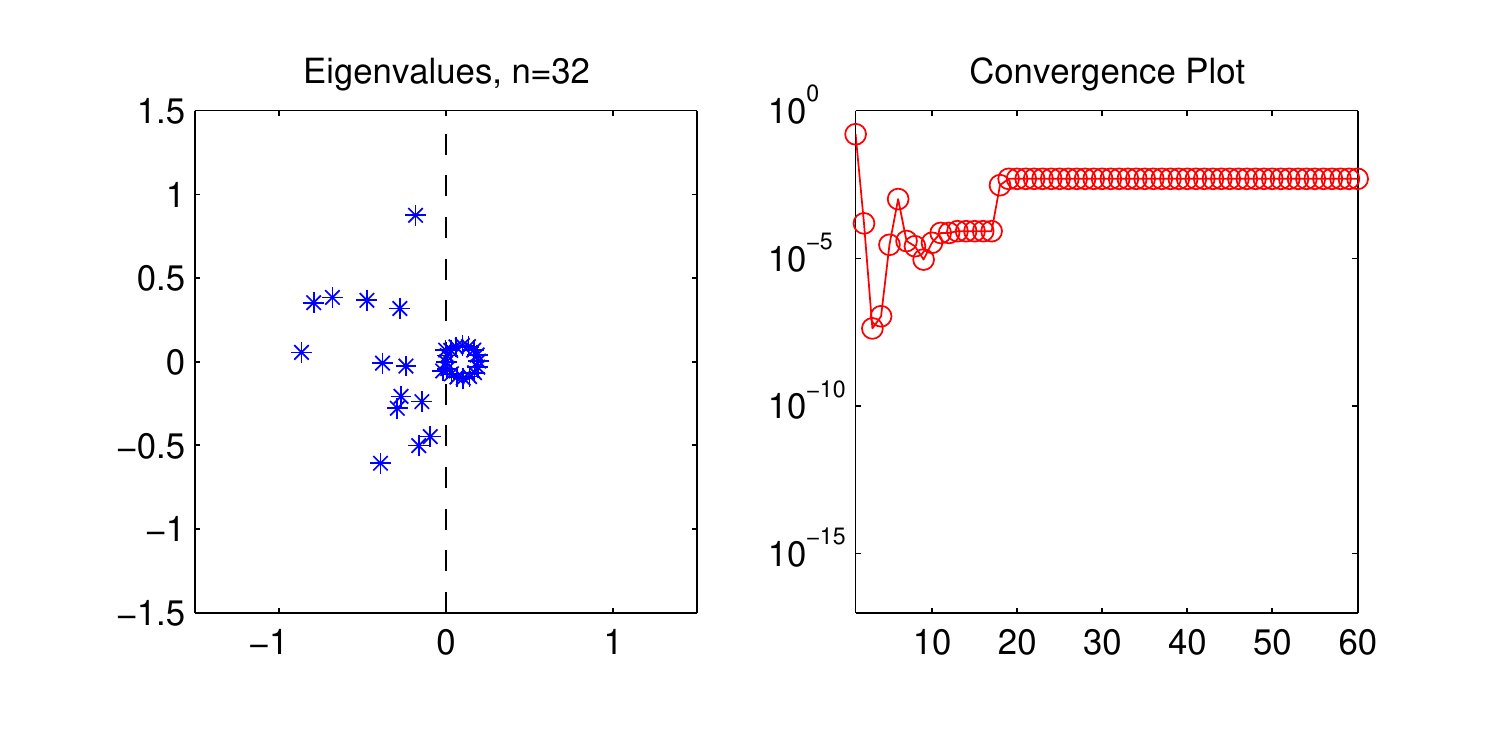}
\label{subfig:numexp_jordan32}
}
\caption{Numerical experiments for Algorithm \textbf{RNEP}.  On the left of each subfigure
is a plot of the eigenvalues in the complex plane (as computed by standard algorithms)
with the imaginary axis (where the split occurs) highlighted.  Note in
\ref{subfig:numexp_imagsplit25}, all eigenvalues are at a distance of $10^{-10}$ from
the imaginary axis.  In \ref{subfig:numexp_jordan32}, the imaginary axis intersects the
$\epsilon$-pseudospectrum (where $\epsilon$ is machine precision).  On the right is a
convergence plot of the backward error as described in Section~\ref{sec:NumExp} for up
to $60$ iterations.}
\label{fig:numexp_hard}
\end{figure}

\subsection{High level strategy for divide-and-conquer, single symmetric matrix (or SVD) case}

In this case, the eigenvalues are on the real line. After obtaining upper and lower bounds for the spectrum (e.g., Gershgorin intervals),  we use vertical lines to split it. In order to try and split a large piece of the spectrum at a time, we will choose the lines to be close to the centers of the obtained intervals. 

From Lemma 3 of \cite{baidemmelgu94} we can see that the basic splitting algorithm will converge 
if the splitting line does not intersect some $\hat{\epsilon}$-pseudospectrum of the matrix, 
which consists of a union of at most $n$ $2\hat{\epsilon}$-length intervals. 
Note that here we only consider symmetric perturbations to the matrix, 
and thus the pseudospectrum is a union of intervals.

To maximize the chances of a good split we will randomize the choice of splitting line, as follows: for an interval \emph{I}, which can without loss of generality be taken as $[-R,R]$, we will choose a point $x$ uniformly between $[-R/2, R/2]$ and use the vertical line passing through $x$ for our split. See Figure \ref{prog_symm} below.

\begin{figure}[ht!]
\begin{center}


\includegraphics[width=6in]{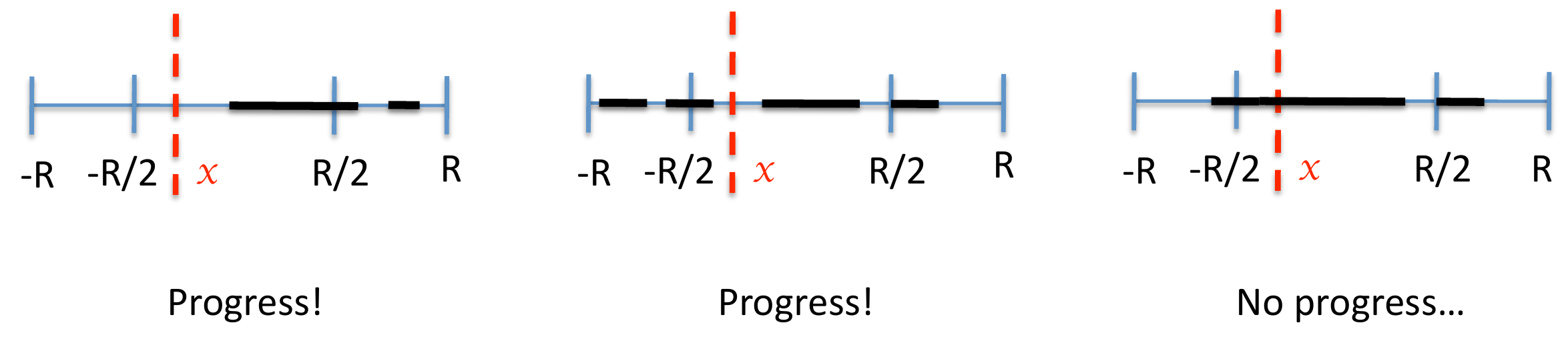}


\caption{The possible outcomes of splitting the segment $[-R, R]$. The thick dark segments represent pseudospectral components (unions of overlapping smaller segments), while the dashed line represents the random choice of split.} \label{prog_symm}
\end{center}
\end{figure}

As before, we measure progress in two ways: one is when a successful split occurs, i.e.,we separate some of the eigenvalues, and the second one is when no split occurs, but a large piece of the interval is determined not to contain any eigenvalues. Thus, progress can fail to occur in this context \emph{only} if our choice of $x$ falls inside one (or more) of the small intervals representing the $\hat{\epsilon}$-pseudospectrum. The total length of these intervals is at most $2 \hat{\epsilon} n$. Therefore, it is a simple calculation to show that the probability that a randomly chosen line in $[-R/2, R/2]$ will intersect the $\hat{\epsilon}$-pseudospectrum is larger than $ 1- 2n \hat{\epsilon}/R$. 

As before, the right choice for $\hat{\epsilon}$ must be proportional to 
$c f(n) ||A||$, with $f(n)$ a polynomial of low degree; 
also as before, $R$ will be $O(||A||)$; putting these two together with 
Lemma \ref{d2}, we get the following. 

\begin{corollary} \label{d4} 
With probability at least $1 - c \cdot \epsilon \cdot n \cdot f(n)$, the number of iterations is at most $ \log_2 \frac{1}{c \epsilon f(n)} + \log_2 \log_2 \frac{1}{\epsilon}$.
\end{corollary}

Finally, we note that a large cluster of eigenvalues actually helps convergence in the symmetric case, 
as the following lemma shows:

\begin{lemma} If $A = \bmat{cc} A_{11} & A_{21}^T \\ A_{21} & A_{22} \emat$ is hermitian, 
$\|A_{21}\|_2 \leq (\lambda_{\max}(A) - \lambda_{\min}(A))/2$.
\end{lemma}
In other words if the eigenvalues are tightly clustered 
($\lambda_{\max}(A) - \lambda_{\min}(A)$ is tiny) then the 
matrix must be nearly diagonal.
\begin{proof}
Choose unit vectors $x$ and $y$ such that $y^T A_{21} x = \| A_{21} \|_2$,
and let $z_{\pm} = 2^{-1/2} \bmat{c} x \\ \pm y \emat$.
Then $z_{\pm}^T A z_{\pm} = .5 x^T A_{11} x + .5 y^T A_{22} y \pm \|A_{21}\|_2$.
So by the Courant-Fischer min-max theorem,
$\lambda_{\max}(A) - \lambda_{\min}(A) \geq z_+^T A z_+ - z_-^T A z_- = 2\|A_{21}\|_2$.
\end{proof}

\subsection{High level strategy for divide-and-conquer, general pencil case}

Now we consider the case of a general square pencil defined by $(A,B)$, and how to reduce it to (block) upper triangular form by dividing-and-conquering the spectrum in an analogous way to the previous sections.  Of course, in this generality the problem is not necessarily well-posed, for if the pencil is {\em singular}, i.e. the characteristic polynomial $\det(A - \lambda B) \equiv 0$ independent of $\lambda$, then {\em every} number in the extended complex plane can be considered an eigenvalue. Furthermore, arbitrarily small perturbations to $A$ and $B$ can make the pencil {\em regular} (nonsingular) with at least one (and sometimes all) eigenvalues located at any desired point(s) in the extended complex plane. (For example, suppose $A$ and $B$ are both strictly upper triangular, so we can change their diagonals to be arbitrarily tiny values with whatever ratios we like.)

The mathematically correct approach in this case is to compute the Kronecker Canonical Form (or rather the variant that only uses orthogonal transformations to maintain numerical stability \cite{demmelkagstrom93a,demmelkagstrom93b}),  a problem that is beyond the scope of this paper.

The challenge of singular pencils is shared by any algorithm for the  generalized eigenproblem. The conventional QZ iteration may converge, but tiny changes in the inputs (or in the level or direction of roundoff) will generally lead to utterly different answers. So ``successful convergence'' does not mean the answer can necessarily be used without further analysis. In contrast, our divide-and-conquer approach will likely simply fail to divide the spectrum when the pencil is close to a singular pencil. In the language of section~\ref{sec:HLS:Nonsymmetric}, for a singular pencil the distance to any splitting line or circle is zero, and the $\epsilon$-pseudospectrum includes the entire complex plane for any $\epsilon > 0$. When such repeated failures occur, the algorithm should simply stop and alert the user of the likelihood of near-singularity.

Fortunately, the set of singular pencils form an algebraic variety of high codimension \cite{demmeledelman95},  i.e. a set of measure zero, so it is unlikely that a randomly chosen pencil will be close to singular. Henceforth we assume that we are not near a singular pencil.

Even in the case of a regular pencil, we can still have eigenvalues that are arbitrarily large, or even equal to infinity, if $B$ is singular. So we cannot expect to find a single finite bounding circle holding all the eigenvalues, for example using Gershgorin disks \cite{Stewart75}, to which we could then apply the techniques of the last section, because  many or all of the eigenvalues could lie outside any finite circle. Indeed, a Gershgorin disk in \cite{Stewart75} has infinite radius as long as there is a common row of $A$ and $B$ that is not diagonally dominant.

In order to exploit the divide-and-conquer algorithm for the interior of circles developed in the last section, we proceed as follows: We begin by trying to divide the spectrum along the unit circle in the complex plane. If all the eigenvalues are inside, we proceed with circles of radius $2^{-2^k}$, i.e. radius $\frac{1}{2}$ and then repeatedly squaring, until we find a smallest such bounding  circle, or (in a finite number of steps) run out of exponent bits in the floating point format (which means all the eigenvalues are in a circle of smallest possible representable radius, and we are done).

Conversely, if all the eigenvalues are outside, we proceed with  circles of radius $2^{2^k}$, i.e. radius $2$ and repeatedly squaring, again until we either find a largest such bounding circle (all eigenvalues being outside), or we run out of exponent bits. At this point, we swap $A$ and $B$, taking the reciprocals of all the eigenvalues, so that they are inside a (small) circle, and proceed as in the last section. If the spectrum splits along a circle of radius $r$, with some inside (say those of $A_{11} - \lambda B_{11}$)  and some outside (those of $A_{22} - \lambda B_{22}$),  we instead compute eigenvalues of $B_{22} - \lambda A_{22}$ inside a circle of radius $\frac{1}{r}$.

\subsection{Related work}


As mentioned in the introduction, there have been many previous approaches to parallelizing eigenvalue/singular value computations, and here we compare and contrast the three closest such approaches to our own.  These three algorithms are the PRISM project algorithm \cite{bischofledermansuntsao}, the reduction of symmetric EIG to matrix multiplication by Yau and Lu \cite{luyau}, and the matrix sign function algorithm originally formulated in \cite{godunov86,bulgakov88,malyshev89,malyshev92,malyshev93}.

At the base of each of the approaches is the idea of computing, explicitly or implicitly, a matrix function $f(A)$  that maps the eigenvalues of the matrix (roughly) into $0$ and $1$, or $0$ and $\infty$.

\subsubsection{The Yau-Lu algorithm}

This algorithm essentially splits the spectrum in the symmetric case (as explained below), but instead of divide-and-conquer, which only aims at reducing the problem into two subproblems of comparable size, the Yau-Lu algorithm aims to break the problem into a large number of small subproblems, which can then be solved by classical methods. 

 The starting idea is put all the eigenvalues on the unit circle by computing $B = e^{iA}$; next, they use a Chebyshev polynomial $P_N(z)$ of high degree ($N \approx 8n$) to evaluate $u(\lambda) = P_N(e^{-i \lambda}B)v_0$ for $\lambda = \frac{i k \pi}{N}$ for  $k = 0, \ldots, N-1$. If  $\lambda$ is close to an eigenvalue, $u(\lambda)$ will be  close to the eigenvector ($P_N(z)$ has a peak at $z=1$, and is very close to $0$  everywhere except a small neighborhood of $z=1$ on the unit circle).  They use the FFT to compute $u(\lambda)$ simultaneously for all $\lambda$,  which makes things a lot faster ($O(n^3\log_2n)$ instead of $O(n^4)$). This works only under the strong assumption that by scaling the eigenvalues so as to be between $0$ and $2\pi$, the gaps between them are all of order about $O(1/n)$. 

Finally, they extract the most likely candidates for eigenvectors from among the $u(\lambda)$, split them into $p$ orthogonal clusters, add more vectors if necessary, construct an orthogonal basis $W$ whose elements span invariant subspaces of $A$, and thus obtain that $W^{T}AW$ decouples the spectrum, making it necessary and sufficient to solve $p$ small eigenvalue problems. 

The authors make an attempt to perform an error analysis as well as a complexity analysis; neither completely addresses the case of tightly clustered eigenvalues. Clustering is a well-known challenge for attaining orthogonality with reasonable complexity, as highlighted in \cite{PD06,DPV06}.

By contrast, in the symmetric case, for highly clustered eigenvalues, our strategy would produce small bounding intervals and output them together with the number of eigenvalues to be found in them, as explained in the previous section. 

\subsubsection{Divide-and-conquer: the PRISM (symmetric) and sign function (both symmetric and non-symmetric) approaches}

The algorithms \cite{bischofledermansuntsao,godunov86,bulgakov88,malyshev93} can be described as follows. Given a matrix $A$:
\begin{enumerate}
\item[Step 1.] Compute, explicitly or implicitly, a matrix function $f(A)$  
that maps the eigenvalues of the matrix (roughly) into $0$ and $1$, or $0$ and $\infty$, 
\item[Step 2.] Do a rank-revealing decomposition (e.g. \textbf{QR}) to find the two spaces 
corresponding to the eigenvalues mapped to $1$ and $0$, or $\infty$ and $0$, respectively;
\item[Step 3.] Block-upper-triangularize the matrix to split the problem into two or more smaller problems;
\item[Step 4.] Recur on the diagonal blocks until all eigenvalues are found.
\end{enumerate}

The strategy is thus to divide-and-conquer; issues of concern have to deal with the stability of the algorithm, the type and cost of the rank-revealing decomposition performed, the convergence of the first step, as well as the splitting strategy. We will try to address some of these issues below.

In the PRISM approach (for the symmetric eigenvalue, and hence by extension the SVD) \cite{bischofledermansuntsao}, the authors choose Beta functions (polynomials which roughly approximate the step function from $[0,1]$ into $\{0,1\}$, with the step taking place at $1/2$) with or without acceleration techniques for their iteration. No stability analysis is performed, although there is reason to believe that some sort of forward stability could be achieved; the rank-revealing decomposition used is \textbf{QR} with pivoting. As such, there is also reason to believe that the approach is communication-avoiding, and that optimality can be reached by employing a lower-bound-achieving \textbf{QR}. 

Since $f(A)$ is computed explicitly, one could potentially use a deterministic rank-revealing \textbf{QR} algorithm instead of our randomized one. There has been recent progress in developing such a \textbf{QR} algorithm that also minimizes communication \cite{BDGGX10}.

The other approach is to use the sign function algorithm introduced by \cite{Roberts}; this works for both the symmetric and the non-symmetric cases, but requires taking inverses, which leads to a potential loss of stability. This can, however, be compensated for by ad-hoc means, e.g., by choosing a slightly different location to split the spectrum. In either case, some assumptions must be made in order to argue stability. 
A more in-depth discussion of algorithms using this method can be found in Section 3 of \cite{baidemmelgu94}. 


\subsubsection{Our algorithm}

Following in the footsteps of \cite{baidemmelgu94} and \cite{DDH07}, we have used the basic algorithm of \cite{godunov86,bulgakov88,malyshev89,malyshev92,malyshev93}.  This algorithm was modified to be numerically stable and otherwise more practical (avoiding matrix exponentials) in \cite{baidemmelgu94}. 
It was then reduced to matrix multiplication and dense \textbf{QR} in \cite{DDH07} by introducing randomization in order to achieve the same arithmetic complexity as fast matrix multiplication (e.g. $O(n^{2.81})$ using Strassen's matrix multiplication). 

The algorithms we present here began as the algorithms in \cite{DDH07} and were modified to take advantage of the optimal communication-complexity $O(n^3)$ algorithms for matrix multiplication and dense \textbf{QR} decomposition (the latter of which was given in \cite{DGHL08}).

We also make here several contributions to the numerical analysis of the methods themselves. First, they depend on a rank-revealing generalized QR decomposition of the product of two matrices $A^{-1}B$,  which is numerically stable and works with high probability. By contrast, this is done with pivoting in \cite{baidemmelgu94}, and the pivot order depends only on $B$, whereas the correct choice obviously depends on $A$ as well. We know of no way to stably compute a rank-revealing generalized QR decomposition of expressions like $A^{-1}B$ without randomization. More generally, we show how to compute a randomized rank revealing QR decomposition of an arbitrary product of matrices (and/or inverses); some of the ideas were used in a deterministic context in G.W. Stewart's paper \cite{Stewart95}, with application to graded matrices. 

Second, we elaborate the divide-and-conquer strategy in considerably more detail than previous publications; the symmetric/SVD and non-symmetric cases differ significantly. The symmetric/SVD case returns a list of eigenvalues/eigenvectors (singular values/vectors) as usual; in the nonsymmetric case, the ultimate output may not be a matrix in Schur form, but rather block Schur form along with a bound (most simply a convex hull) on the $\epsilon$-pseudospectrum of each diagonal block, along with an indication that further refinement of the pseudospectrum is impossible without higher precision. For example, a single Jordan block would be represented (roughly) by a circle centered at the eigenvalue. A conventional algorithm would of course return $n$ eigenvalues evenly spaced along the circular border of the pseudospectrum, with no indication (without further computation) that this is the case.  Both basic methods are oblivious to the genus of the pseudospectrum, as they will fail to point out ``holes'' like the one present in the third picture of Figure \ref{pr_nopr}. In certain cases (like convex or almost convex pseudospectra) the information returned by divide-and-conquer may be more useful than the conventional method.

We emphasize that if a subproblem is small enough to solve by the standard algorithm using minimal communication, a list of its eigenvalues will be returned. So we will only be forced to return bounds on pseudospectra when the dimension of the subproblem is very large, e.g. does not fit in cache in the sequential case.   

\section{Communication Bounds for Computing the Schur Form} \label{cba}

\subsection{Lower bounds}

We now provide two different approaches to proving communication 
lower bounds for computing the Schur decomposition.  Although 
neither approach yields a completely satisfactory algorithm-independent
proof of a lower bound, we discuss the merits and drawbacks as well as 
the obstacles associated with each approach.

The goal of the first approach is to show that computing the 
QR decomposition can be reduced to computing the Schur form 
of a matrix of nearly the same size.  Such a reduction in
principle provides a way to extend known lower bounds 
(communication or arithmetic) for QR decomposition to 
Schur decomposition, but has limitations described below.

We outline below a reduction of QR to a Schur decomposition 
followed by a triangular solve with multiple right hand sides (TRSM).  
There are two main obstacles to the success of this approach.  
First, the TRSM may have a nontrivial cost so that proving the 
cost of QR is at most the cost of computing Schur plus the cost 
of TRSM is not helpful.  Second, all known proofs of lower bounds for 
QR apply to a restricted set of algorithms which satisfy certain 
assumptions, and thus the reduction is only applicable to 
Schur decomposition and TRSM algorithms which also satisfy those 
assumptions.

The first obstacle can be overcome with additional assumptions which 
are made explicit below.  We do not see a way to overcome the 
second obstacle with the assumptions made in the current known 
lower bound proofs for QR decomposition, but future, more general 
proofs may be more amenable to the reduction.

\ignore{
Here we give a simple reduction that shows that the cost of computing the Schur form of a matrix is bounded below by the cost of computing the QR decomposition, under certain assumptions made clear below.
}

We now give the reduction of QR decomposition to Schur decomposition and TRSM.  Let $R$ be $m$-by-$m$ and upper triangular, $X$ be $n$-by-$m$ and dense, $A = \bmat{c} R \\ X \emat$ be $(m+n)$-by-$m$, and $B = \bmat{cc} R & 0^{m \times n} \\ X & 0^{n \times n} \emat$ be $(m+n)$-by-$(m+n)$. Let 
\[
A = \hat{Q} \cdot R_A = \bmat{cc} Q_{11} & Q_{12} \\ Q_{21} & Q_{22} \emat \cdot
\bmat{c} \hat{R} \\ 0^{n \times m} \emat
\]
where $\hat{R}$ and $Q_{11}$ are both $m$-by-$m$, and $Q_{22}$ is $n$-by-$n$. Note that $Q_{11}$ and $\hat{R}$ are both upper triangular. Then it is easy to confirm that the Schur decomposition of $B$ is given by
\[
 \hat{Q}^T \cdot B \cdot \hat{Q} = 
     \bmat{cc} \hat{R} \cdot Q_{11} & \hat{R} \cdot Q_{12}
   \\ 0^{n \times m} & 0^{n \times n} \emat \equiv
\bmat{cc} T_{11} & T_{12} \\ 0^{n \times m} & 0^{n \times n} \emat
\]

Since $T_{11}$ is the upper triangular product of $\hat{R}$ and $Q_{11}$, 
we can easily solve for $\hat{R} = T_{11} \cdot Q_{11}^{-1}$. 
This leads to our reduction algorithm \textbf{S2QR} of the QR decomposition of $A$ to 
Schur decomposition of $B$ outline below. Note that here \textbf{Schur} 
stands for any black-box algorithm for computing the respective decomposition. 
In addition, \textbf{TRSM} stands for ``triangular solve with multiple right-hand-sides'', 
which costs $O(m^3)$ flops, and for which communication-avoiding implementations are known \cite{BDHS10}.

\begin{algorithm}[ht!]
\protect\caption{Function $[\hat{Q}, \hat{R}] = $\textbf{S2QR}$([R; X])$, where $R$ is $n \times n$ upper triangular and $X$ is $m \times n$}
\label{qr2schur}
\begin{algorithmic}[1]
\STATE Form $B = \left [ \begin{array}{cc} R & 0 \\ X & 0 \end{array} \right ]$
\STATE Compute $[\hat{Q}, T] = $\textbf{Schur}$(B)$
\STATE $Q_{11} = \hat{Q}(1:m, 1:m)$; $T_{11} = T(1:m, 1:m)$
\STATE $\hat{R} = $\textbf{TRSM}$[Q_{11}, T_{11}]$
\RETURN $[\hat{Q}, \hat{R}]$
\end{algorithmic}
\end{algorithm}

This lets us conclude that
\[
{\rm Cost}((m+n)\times(m+n) \; \mbox{\textbf{Schur}}) \geq
{\rm Cost}((m+n)\times(m) \: \mbox{\textbf{S2QR}}) -
{\rm Cost}((m)\times(m) \: \mbox{\textbf{TRSM}}) 
\]

We assume now that we are using an ``$O(n^3)$ style'' algorithm, which means that the cost of \textbf{S2QR}
is $\Omega (m^2 n)$, say at least $c_1 m^2 n$ for some constant $c_1$. Similarly, the cost of \textbf{TRSM} is $O(m^3)$, say at most $c_2 m^3$. We choose $m$ to be a suitable fraction of $n$, say $m = \min(.5 c_1/c_2,1)n$, so that
\[
{\rm Cost}((2n)\times(2n) \; {\rm \textbf{Schur}}) \geq .5 c_1 m^2 n = \Omega(n^3)
\]
as desired.

We note that this proof can provide several kinds of lower bounds: \begin{itemize} \item[(1)] It provides a lower bound on arithmetic operations, by choosing $c_1$ to bound the number of arithmetic operations in \textbf{S2QR} from below \cite{DGHL08}, 
and $c_2 = 1/3$ to bound the number of arithmetic operations in \textbf{TRSM} from above \cite{BDHS10}.  
\item[(2)] It can potentially provide a lower bound on the number of words moved (or the number of messages) in the sequential case, by taking $c_1$ and $c_2$ 
both suitably proportional to $1/M^{1/2}$ (or to $1/M^{3/2}$). 
We note that the communication lower bound proof in \cite{BDHS10}  
does {\em not} immediately apply to this algorithm, 
because it assumed that the QR decomposition was computed 
(in the conventional backward-stable fashion) by repeated pre-multiplications 
of $A$ by orthogonal transformation, or by Gram-Schmidt, or by Cholesky-QR.
However, this does not invalidate the above reduction.
\item[(3)] It can similarly provide a lower bound for all communication in the parallel case, 
assuming that the arithmetic work of both \textbf{S2QR} and \textbf{TRSM} 
are load balanced, by taking $c_1$ and $c_2$ both proportional to $1/P$. 
\end{itemize}


The second approach to proving a communication lower bound for Schur decomposition 
is to apply the lower bound result for applying orthogonal transformations given by
Theorem~4.2 in \cite{BDHS10}.  We must assume that the algorithm for computing 
Schur form performs $O(n^3)$ flops which satisfy the assumptions of this result, 
to obtain the desired communication lower bound.  This approach is a 
straightforward application of a previous result, but the main drawback is the 
loss of generality.  Fortunately, this lower bound does apply to the 
spectral divide-and-conquer algorithm (by applying it just to the QR decompositions
performed) as well as the first phase of the successive band reduction approach 
discussed in Section~\ref{sec_SBR};
this is all we need for a lower bound on the entire algorithm.
But it does not necessarily apply to all future algorithms for Schur form.

The results of Theorem 4.2 in \cite{BDHS10} require that a sequence of 
Householder transformations are applied to a matrix (perhaps in a blocked fashion) 
where the Householder vectors are computed to annihilate entries of a 
(possibly different) matrix such that \emph{forward progress} is maintained,
that is, no entry which has been previously annihilated is filled in by a 
later transformation.  Therefore, this lower bound does not apply to 
``bulge-chasing,'' the second phase of successive band reduction,
where a banded matrix is reduced to tridiagonal form.
But it does apply to the reduction of a 
full symmetric matrix to banded form via QR decomposition and two-sided 
updates, the first phase of the
successive band reduction algorithms in Section~\ref{sec_SBR}.

Reducing a full matrix to banded form (of bandwidth $b$) requires 
$F=\frac43 n^3 + bn^2 -\frac83 b^2n + \frac13 b^3$ flops, 
so assuming forward progress, the number of words moved is at least 
$\Omega\lt(\frac{F}{\sqrt M}\rt)$, so for $b$ at most some fraction 
of $n$, we obtain a communication lower bound of $\Omega\lt(\frac{n^3}{\sqrt M}\rt)$. 

\subsection{Upper bounds: sequential case} 
\label{USeq}

In this section we analyze the complexity of the communication and computation of the algorithms described in Section~\ref{Algs} when performed on a sequential machine.  We define $M$ to be the size of the fast memory, $\alpha$ and $\beta$ as the message latency and inverse bandwidth between fast and slow memory, and $\gamma$ as the flop rate for the machine.  We assume the matrix is stored in contiguous blocks of optimal blocksize $b\times b$, where $b=\Theta(\sqrt M)$.  We will ignore lower order terms in this analysis, and we assume the original problem is too large to fit into fast memory (i.e., $an^2 > M$ for $a=1,2$ or $4$).
\subsubsection{Subroutines}
\label{USeqSub}

The usual blocked matrix multiplication algorithm for multiplying square matrices of size $n\times n$ has a total cost of $$C_{MM}(n) = \alpha \cdot O\lt(\frac{n^3}{M^{3/2}}\rt) + \beta \cdot O\lt( \frac{n^3}{\sqrt M}\rt) + \gamma \cdot O \lt( n^3\right).$$

From \cite{DGHL08}, the total cost of computing the QR decomposition using the so-called Communication-Avoiding QR algorithm (\textbf{CAQR}) of an $m\times n$ matrix ($m\geq n$) is $$C_{QR}(m,n) = \alpha \cdot O \lt( \frac{mn^2}{M^{3/2}}\rt) + \beta \cdot O\lt( \frac{mn^2}{\sqrt M}\rt) + \gamma \cdot O \lt(mn^2\rt)~.$$ This algorithm overwrites the upper half of the input matrix with the upper triangular factor $R$, but does not return the unitary factor $Q$ in matrix form (it is stored implicitly).  To see that we may compute $Q$ explicitly with only a constant factor more communication and work, we could perform \textbf{CAQR} on the following block matrix:
$$\left[\begin{matrix} A & I \\ 0 & 0 \end{matrix}\right] = \left[\begin{matrix} Q & 0 \\ 0 & I \end{matrix}\right] \left[\begin{matrix} R & Q^H \\ 0 & 0 \end{matrix}\right].$$

In this case, the block matrix is overwritten by an upper triangular matrix that contains $R$ and the conjugate-transpose of $Q$. For $A$ of size $n\times n$, this method incurs a cost of $C_{QR}(2n,2n) = O( C_{QR}(n,n) )$. (This is not the most practical way to compute $Q$, but is sufficient for a $O( \cdot )$ analysis.)  We note that an analogous algorithm (with identical cost) can be used to obtain \textbf{RQ}, \textbf{QL}, and \textbf{LQ} factorizations.

Computing the randomized rank-revealing QR decomposition (Algorithm~\ref{rurv}) consists of generating a random matrix, performing two \textbf{QR} factorizations (with explicit unitary factors), and performing one matrix multiplication.  Since the generation of a random matrix requires no more than $O(n^2)$ computation or communication, the total cost of the \textbf{RURV} algorithm is
\begin{eqnarray*}
C_{RURV}(n) &=& 2\cdot C_{QR} (n,n) + C_{MM}(n) \\
&=&  \alpha \cdot O \lt(\frac{n^3}{M^{3/2}}\rt) + \beta \cdot O \lt(\frac{n^3}{\sqrt M}\rt) + \gamma \cdot O \lt(n^3\rt).
\end{eqnarray*}
We note that this algorithm returns the unitary factor $U$ and random orthogonal factor $V$ in explicit form.  The analogous algorithm used to obtain the \textbf{RULV} decomposition has identical cost.

Performing implicit repeated squaring (Algorithm~\ref{rsq}) consists of two matrix additions followed by iterating a \textbf{QR} decomposition of a $2n\times n$ matrix followed by two $n\times n$ matrix multiplications until convergence.  Checking convergence requires computing norms of two matrices, constituting lower order terms of both computation and communication. Since the matrix additions also require lower order terms and assuming convergence after some constant $p$ iterations (depending on $\epsilon$, the accuracy desired, as per Corollaries~\ref{d3} and \ref{d4}), the total cost of the \textbf{IRS} algorithm is
\begin{eqnarray*}
C_{IRS}(n) &=& p \cdot \lt( C_{QR}(2n,n) + 2\cdot C_{MM}(n) \rt)\\
&=& \alpha \cdot O \lt(\frac{n^3}{M^{3/2}}\rt) + \beta \cdot O \lt(\frac{n^3}{\sqrt M}\rt) + \gamma \cdot O \lt(n^3\rt).
\end{eqnarray*}

Choosing the dimensions of the subblocks of a matrix in order to minimize the 1-norm of the lower left block (as in computing the $n-1$ possible values of $\|E_{21}\|_1/\|A\|_1+\|F_{21}\|_1/\|B\|_1$ in the last step of Algorithm~\ref{rgnep}) can be done with a blocked column-wise prefix sum followed by a blocked row-wise max reduction (the Frobenius and max-row-sum norms can be computed with similar two-phase reductions).  This requires $O(n^2)$ work, $O(n^2)$ bandwidth, and $O\lt(\frac{n^2}{M}\rt)$ latency, as well as $O(n^2)$ extra space.  Thus, the communication costs are lower order terms when $an^2 > M$.

 \subsubsection{Randomized Generalized Non-symmetric Eigenvalue Problem (\textbf{RGNEP})}

We consider the most general divide-and-conquer algorithm; the analysis of the other three algorithms differ only by constants.  One step of \textbf{RGNEP} (Algorithm~\ref{rgnep}) consists of 2 evaluations of \textbf{IRS}, 2 evaluations of \textbf{RURV}, 2 evaluations of \textbf{QR} decomposition, and 6 matrix multiplications, as well as some lower order work.  Thus, the cost of one step of \textbf{RGNEP} (not including the cost of subproblems) is given by
 \begin{eqnarray*}
C_{RGNEP^*}(n) &=& 2 \cdot C_{IRS}(n) + 2 \cdot C_{RURV}(n) + 2 \cdot C_{QR}(n,n) + 6 \cdot C_{MM}(n) \\
&=& \alpha \cdot O \lt(\frac{n^3}{M^{3/2}}\rt) + \beta \cdot O \lt(\frac{n^3}{\sqrt M}\rt) + \gamma \cdot O \lt(n^3\rt).
\end{eqnarray*}
Here and throughout we denote by $C_{\text{ALG}^*}$ the cost of a single divide-and-conquer step in executing \textbf{ALG}.

Assuming we split the spectrum by some fraction $f$, one step of RGNEP creates two subproblems of size $fn$ and $(1-f)n$. 
If we continue to split by fractions in the range $[1-f_0,f_0]$ at each step, for some threshold $\frac12\leq f_0<1$, the total cost of the RGNEP algorithm is bounded by the recurrence $$C(n) = \left\{ \begin{array}{lr}
C\lt(f_0n\rt) + C\lt((1-f_0)n\rt) + C_{RGNEP^*}(n) & \text{if}\;\; 2n^2 > M  \\
\alpha \cdot O (1) + \beta \cdot O (n^2) + \gamma \cdot O \lt(n^3\rt) & \text{if}\;\; 2n^2 \leq M
\end{array} \right.$$
where the base case arises when input matrices $A$ and $B$ fit into fast memory and the only communication is to read the inputs and write the outputs.  This recurrence has solution
$$ C(n) = \alpha \cdot O \lt(\frac{n^3}{M^{3/2}}\rt) + \beta \cdot O \lt(\frac{n^3}{\sqrt M}\rt) + \gamma \cdot O \lt(n^3\rt). $$

Note that while the cost of the entire problem is asymptotically the same as the cost of the first divide-and-conquer step, the constant factor for $C(n)$ is bounded above by $\frac{1}{3f_0(1-f_0)}$ times the constant factor for $C_{RGNEP*}(n)$.  That is, for constant $A$ defined such that $C_{RGNEP*}(n)\leq An^3$, we can verify that $C(n)\leq \frac{A}{3f_0(1-f_0)} n^3$ satisfies the recurrence:
\begin{eqnarray*}
C(n) &\leq& C(f_0n)+C((1-f_0)n)+C_{RGNEP*}(n) \\
&\leq& \frac{A}{3f_0(1-f_0)} (f_0n)^3 + \frac{A}{3f_0(1-f_0)} ((1-f_0)n)^3 + An^3 \\
&\leq& \frac{(f_0^3+(1-f_0)^3)An^3+(3f_0(1-f_0))An^3}{3f_0(1-f_0)} \\
&\leq& \frac{A}{3f_0(1-f_0)}n^3.
\end{eqnarray*}

\subsection{Upper bounds: parallel case} 
\label{UPar}

In this section we analyze the complexity of the communication and computation of the algorithms described in Section~\ref{Algs} when performed on a parallel machine.  We define $P$ to be the number of processors, $\alpha$ and $\beta$ as the message latency and inverse bandwidth between any pair of processors, and $\gamma$ as the flop rate for each processor.  We assume the matrices are distributed in a 2D blocked layout with square blocksize $b=\frac{n}{\sqrt P}$.  As before, we will ignore lower order terms in this analysis.  After analysis of one step of the divide-and-conquer algorithms, we will describe how the subproblems can be solved independently.  

\subsubsection{Subroutines}

The \textbf{SUMMA} matrix multiplication algorithm \cite{SUMMA} for multiplying square matrices of size $n\times n$, with blocksize $b$, has a total cost of
$$C_{MM}(n,P) = \alpha \cdot O\lt( \frac nb \log P \rt) + \beta \cdot O\lt( \frac{n^2}{\sqrt P} \log P\rt) + \gamma \cdot O\lt( \frac{n^3}{P} \rt).$$

From \cite{DGHL08} (equation (12) on page 62), the total cost of computing the QR decomposition using \textbf{CAQR} of an $n\times n$ matrix, with blocksize $b$, is
$$C_{QR}(n,n,P)=\alpha\cdot O\lt(\frac nb \log P \rt)+\beta\cdot O\lt(\frac{n^2}{\sqrt P}\log P\rt)+\gamma\cdot O\lt(\frac{n^3}{P}+\frac{n^2b}{\sqrt P}\log P \rt).$$
The cost of computing the \textbf{QR} decomposition of a tall skinny matrix of size $2n\times n$ is asymptotically the same.

As in the sequential case, \textbf{RURV} and \textbf{IRS} consist of a constant number of subroutine calls to the \textbf{QR} and matrix multiplication algorithms, and thus they have the following asymptotic complexity (here we assume a blocked layout, where $b=\frac{n}{\sqrt  P}$):
$$C_{RURV}(n,P) = \alpha \cdot O\lt( \sqrt P \log P \rt) + \beta \cdot O\lt( \frac{n^2}{\sqrt P} \log P\rt) + \gamma \cdot O\lt( \frac{n^3}{P} \log P  \rt),$$
and
$$C_{IRS}(n,P) = \alpha \cdot O\lt( \sqrt P \log P \rt) + \beta \cdot O\lt( \frac{n^2}{\sqrt P} \log P\rt) + \gamma \cdot O\lt( \frac{n^3}{P} \log P  \rt).$$

Choosing the dimensions of the subblocks of a matrix in order to minimize the 1-norm of the lower left block (as in computing the $n-1$ possible values of $\|E_{21}\|_1/\|A\|_1+\|F_{21}\|_1/\|B\|_1$ in the last step of Algorithm~\ref{rgnep}) can be done with a parallel prefix sum along processor columns followed by a parallel prefix max reduction along processor rows (the Frobenius and max-row-sum norms can be computed with similar two-phase reductions).  This requires $O\lt(\log P\rt)$ messages, $O\lt(\frac{n}{\sqrt P}\log P\rt)$ words, and $O\lt(\frac{n^2}{P}\log P\rt)$ arithmetic, which are all lower order terms, as well as $O\lt(\frac{n^2}{P}\rt)$ extra space.

\subsubsection{Randomized Generalized Non-symmetric Eigenvalue Problem (RGNEP)}

Again, we consider the most general divide-and-conquer algorithm; the analysis of the other three algorithms differ only by constants.  One step of {\bf RGNEP} (Algorithm~\ref{rgnep}) requires a constant number of the above subroutines, and thus the cost of one step of {\bf RGNEP} (not including the cost of subproblems) is given by
$$C_{RGNEP^*}(n,P) = \alpha \cdot O\lt( \sqrt P \log P \rt) + \beta \cdot O\lt( \frac{n^2}{\sqrt P} \log P\rt) + \gamma \cdot O\lt( \frac{n^3}{P}\log P  \rt).$$

Assuming we split the spectrum by some fraction $f$, one step of {\bf RGNEP} creates two subproblems of sizes $fn$ and $(1-f)n$.  Since the problems are independent, we can assign one subset of the processors to one subproblem and another subset of processors to the second subproblem.  For simplicity of analysis, we will assign $f^2P$ processors to the subproblem of size $fn$ and $(1-f)^2P$ processors to the subproblem of size $(1-f)n$.  This assignment wastes $2f(1-f)P$ processors ($\frac12$ of the processors if we split evenly), but this will only affect our upper bound by a constant factor.  

Because of the blocked layout, we assign processors to subproblems based on the data that resides in their local memories.  Figure~\ref{fig:parallelsplit} shows this assignment of processors to subproblems where the split is shown as a dotted line.  Because the processors colored light gray already own the submatrices associated with the larger subproblem, they can begin computation without any data movement.  The processor owning the block of the matrix where the split occurs is assigned to the larger subproblem and passes the data associated with the smaller subproblem to an idle processor.  This can be done in one message and is the only communication necessary to work on the subproblems independently.

\begin{figure}
\centering
\includegraphics[scale=.75]{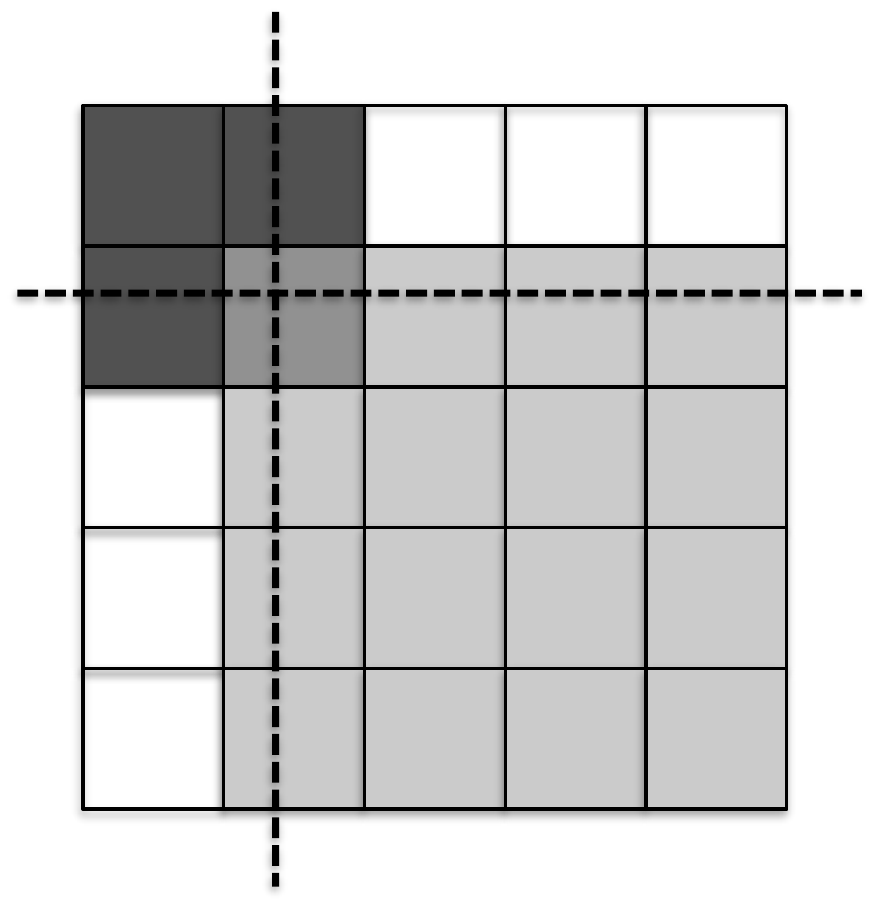}
\caption{Assignment of processors to subproblems.  The solid lines represent the blocked layout on 25 processors, the dotted line represents the subproblem split, and processors are color-coded for the subproblem assignment.  One idle (white) processor is assigned to the smaller subproblem and receives the corresponding data from the processor owning the split.}
\label{fig:parallelsplit}
\end{figure}

If we continue to split by fractions in the range $[1-f_0,f_0]$ at each step, for some threshold $\frac12 \leq f_0<1$, we can find an upper bound of the total cost along the critical path by accumulating the cost of only the larger subproblem at each step (the smaller problem is solved in parallel and incurs less cost).  Although more processors are assigned to the larger subproblem, the approximate identity (ignoring logarithmic factors)
$$C_{RGNEP*}(fn,f^2P) \approx f \cdot C_{RGNEP*}(n,P)$$
implies that the divide-and-conquer step of the larger subproblem will require more time than the corresponding step for the smaller subproblem.  If we assume that the larger subproblem is split by the largest fraction $f_0$ at each step, then the total cost of each smaller subproblem (including subsequent splits) will never exceed that of its corresponding larger subproblem.  Thus, an upper bound on the total cost of  the {\bf RGNEP} algorithm is given by the recurrence
$$C(n,P) = \left\{ \begin{array}{lr}
C\lt(f_0n,{f_0}^2P\rt) + C_{RGNEP^*}(n,P) & \text{if}\;\; P>1  \\
 \gamma \cdot O \lt(n^3\rt) & \text{if}\;\; P=1
\end{array} \right.$$
where the base case arises when one processor is assigned to a subproblem (requiring local computation but no communication).  The solution to this linear recurrence is
$$C(n,P) = \alpha \cdot O\lt( \frac nb \log P \rt) + \beta \cdot O\lt( \frac{n^2}{\sqrt P} \log P\rt) + \gamma \cdot O\lt( \frac{n^3}{P}+\frac{n^2b}{\sqrt P}\log P  \rt).$$
Here the constant factor for the upper bound on the total cost of the algorithm is $\frac{1-P^{-1/2}}{1-f_0}$ times larger than the constant factor for the cost of one step of divide-and-conquer.  This constant arises from the summation $\sum_{i=0}^d {f_0}^i$ where $d=\log_{f_0} \frac{1}{\sqrt P}$ is the depth of the recurrence (we ignore the decrease in the size of the logarithmic factors).

Choosing the blocksize to be $b=\frac{n}{\sqrt P \log P}$, we obtain a total cost for the {\bf RGNEP} algorithm of 
$$C(n,P) = \alpha \cdot O\lt( \sqrt P \log^2 P \rt) + \beta \cdot O\lt( \frac{n^2}{\sqrt P} \log P\rt) + \gamma \cdot O\lt( \frac{n^3}{P} \rt).$$
We will argue in the next section that this communication complexity is within polylogarithmic factors of optimal.

\section{Computing Eigenvectors of a Triangular Matrix} \label{cetm}

After obtaining the Schur form of a nonsymmetric matrix $A=QTQ^*$, finding the eigenvectors of $A$ requires computing the eigenvectors of the triangular matrix $T$ and applying the unitary transformation $Q$.  Assuming all the eigenvalues are distinct, we can solve the equation $TX=XD$ for the upper triangular eigenvector matrix $X$, where $D$ is a diagonal matrix whose entries are the diagonal of $T$.  This implies that for $i<j$,
\begin{equation}
\label{eqn:trevc}
X_{ij} = \frac{\dis T_{ij}X_{jj} + \sum_{k=i+1}^{j-1} T_{ik}X_{kj}}{T_{jj}-T_{ii}}
\end{equation}
where $X_{jj}$ can be arbitrarily chosen for each $j$.  We will follow the LAPACK naming scheme and refer to algorithms that compute the eigenvectors of triangular matrices as {\bf TREVC}.

\subsection{Lower bounds}

Equation~\ref{eqn:trevc} matches the form specified in \cite{BDHS10}, where the $g_{ijk}$ functions are given by the multiplications $T_{ik}\cdot X_{kj}$.  Thus we obtain a lower bound of the communication costs of any algorithm that computes these $O(n^3)$ multiplies.  That is, on a sequential machine with fast memory of size $M$, the bandwidth cost is
$\Omega\lt( \frac{n^3}{\sqrt M} \rt)$
and the latency cost is
$\Omega\lt( \frac{n^3}{M^{3/2}} \rt).$
On a parallel machine with $P$ processors and local memory bounded by $O\lt(\frac{n^2}{P}\rt)$, the bandwidth cost is
$\Omega\lt( \frac{n^2}{\sqrt P} \rt)$
and the latency cost is
$\Omega\lt( \sqrt P \rt).$

\subsection{Upper bounds: sequential case}

The communication lower bounds can be attained in the sequential case by a blocked iterative algorithm, presented in Algorithm~\ref{alg:blck_trevc}. The problem of computing the eigenvectors of a triangular matrix, as well as the formulation of a blocked algorithm, appears as a special case of the computations presented in \cite{Henry95}; however,  \cite{Henry95} does not try to pick a blocksize in order to minimize communication, nor does it analyze the communication complexity of the algorithm.

 For simplicity we assume the blocksize $b$ divides $n$ evenly, and we use the notation $X[i,j]$ to refer to the $b\times b$ block of $X$ in the $i^{th}$ block row and $j^{th}$ block column.  In this section and the next we consider the number of flops, words moved, and messages moved separately and use the notation $A_{ALG}$ to denote the arithmetic count, $B_{ALG}$ to denote to the bandwidth cost or word count, and $L_{ALG}$ to denote the latency cost or message count.  

\begin{algorithm}
\protect\caption{Blocked Iterative {\bf TREVC}}
\label{alg:blck_trevc}
\begin{algorithmic}[1]
\REQUIRE $T$ is upper triangular, $D$ is diagonal with diagonal entries of $T$; \\
	all matrices are blocked with blocksize $b\leq \sqrt{M/3}$ which divides $n$ evenly
\FOR{$j=1$ to $n/b$} 
	\STATE solve $T[j,j]*X[j,j] = X[j,j]*D[j,j]$ for $X[j,j]$
	\COMMENT{read $T[j,j]$, write $X[j,j]$}
	\FOR{$i=j-1$ down to $1$}
		\STATE $S = 0$
		\FOR{$k=i+1$ to $j$}
			\STATE $S = S + T[i,k]*X[k,j]$
			\COMMENT{read $T[i,k]$, read $X[k,j]$}
		\ENDFOR
		\COMMENT{read $T[i,i]$, read $D[j,j]$, write $X[i,j]$}
	\ENDFOR
\ENDFOR
\ENSURE $X$ is upper triangular and $TX=XD$
\end{algorithmic}
\end{algorithm}

Since each block has $O(b^2)$ words, an upper bound on the total bandwidth cost of Algorithm~\ref{alg:blck_trevc} is given by the following summation and inequality:
\begin{eqnarray*}
B_{TREVC}(n) &=& \sum_{j=1}^{n/b} \lt[ O(b^2) + \sum_{i=1}^{j-1} \lt[ \sum_{k=i+1}^j \lt[ O(b^2) \rt] + O(b^2) \rt] \rt] \\
 &\leq& \frac nb \lt[ O(b^2) + \frac nb \lt[ \frac nb \lt[ O(b^2) \rt] + O(b^2) \rt] \rt] \\
 &=& O\lt( \frac{n^3}{b} + n^2 + nb \rt).
 \end{eqnarray*}

If each block is stored contiguously, then each time a block is read or written, the algorithm incurs a cost of one message.  Thus an upper bound on the total latency cost of Algorithm~\ref{alg:blck_trevc} is given by
\begin{eqnarray*}
L_{TREVC}(n) &=& \sum_{j=1}^{n/b} \lt[ O(1) + \sum_{i=1}^{j-1} \lt[ \sum_{k=i+1}^j \lt[ O(1) \rt] + O(1) \rt] \rt] \\
 &=& O\lt( \frac{n^3}{b^3} + \frac{n^2}{b^2} + \frac nb \rt).
 \end{eqnarray*}
 Setting $b=\Theta(\sqrt M)$, such that $b\leq \sqrt{M/3}$, we attain the lower bounds above (assuming $n^2 > M$).

 \subsection{Upper Bounds: Parallel Case}

 There exists a parallel algorithm, presented in Algorithm~\ref{alg:ptrevc}, requiring $O(n^2 / P)$ local memory that attains the communication lower bounds given above to within polylogarithmic factors.  We assume a blocked distribution of the triangular matrix $T$ ($b=n/\sqrt P$) and the triangular matrix $X$ will be computed and distributed in the same layout.  We also assume that $P$ is a perfect square.  Each processor will store a block of $T$ and $X$, three temporary blocks $T'$, $X'$, and $S$, and one temporary vector $D$.  That is, processor $(i,j)$ stores $T_{ij}$, $X_{ij}$, $T_{ij}'$, $X_{ij}'$, $S_{ij}$, and $D_{ij}$.

 The algorithm iterates on block diagonals, starting with the main block diagonal and finishing with the top right block, and is ``right-looking'' (i.e. after the blocks along a diagonal are computed, information is pushed up and to the right to update the trailing matrix).  Blocks of $T$ are passed right one processor column at a time, but blocks of $X$ must be broadcast to all processor rows at each step.

 \begin{algorithm}
\protect\caption{Parallel Algorithm \textbf{PTREVC}}
\label{alg:ptrevc}
\begin{algorithmic}[1]
\REQUIRE $T$ is upper triangular and distributed in blocked layout, $D$ is diagonal with diagonal entries of $T$ (no extra storage needed)
\FORALL{processors $(i,j)$ such that $i\geq j$}
	\STATE set $T_{ij}'=T_{ij}$
	\STATE set $S_{ij}=0$
\ENDFOR
\FORALL{processors $(j,j)$}
	\STATE solve $T_{jj}X_{jj}=X_{jj}D_{jj}$ for $X_{jj}$ locally
	\label{line:diagblocksolve}
	\STATE broadcast $X_{jj},D_{jj}$ up processor column, store in local $X',D'$
	\label{line:diagbroadcast}
\ENDFOR
\FOR{$k=1$ to $\sqrt P - 1$}
\label{line:iterativefor}
	\FORALL{processors $(i,j)$ such that $i-j \geq k-1$}
		\IF{$k \leq j < \sqrt P$}
			\STATE send $T_{ij}' \ra T_{i,j+1}'$ (to right)
			\label{line:sendright}
		\ENDIF
		\IF{$k < j \leq \sqrt P$}
			\STATE receive $T_{i,j-1}' \ra T_{ij}'$ (from left)
			\label{line:recvleft}
		\ENDIF
		\IF{$i-j > k-1$}
			\STATE update $S_{ij} = S_{ij} + T_{ij}'X_{ij}'$ locally
			\label{line:localupdate}
		\ENDIF
	\ENDFOR
	\FORALL{processors $(i,j)$ such that $i-j=k$}
		\STATE solve $T_{ij}'X_{ij}+S_{ij}=X_{ij}D_{ij}'$ for $X_{ij}$ locally
		\label{line:offdiagblocksolve}
		\STATE broadcast $X_{ij}$ up processor column, store in local $X'$
		\label{line:offdiagbroadcast}
	\ENDFOR
\ENDFOR
\ENSURE $X$ is upper triangular and distributed in blocked layout, $TX=XD$
\end{algorithmic}
\end{algorithm}

First we consider the arithmetic complexity.  Arithmetic work occurs at lines~\ref{line:diagblocksolve}, \ref{line:localupdate}, and \ref{line:offdiagblocksolve}.  All of the for loops in the algorithm, except for the one starting at line~\ref{line:iterativefor}, can be executed in parallel, so the arithmetic cost of Algorithm~\ref{alg:ptrevc} is given by
\begin{eqnarray*}
A_{TREVC}(n,P) &=& O\lt(\lt( \frac{n}{\sqrt P} \rt)^3\rt) + \sqrt P \lt( O\lt(\lt( \frac{n}{\sqrt P} \rt)^3\rt) + O\lt(\lt( \frac{n}{\sqrt P} \rt)^3\rt) \rt) \\
 &=& O\lt(\frac{n^3}{P}\rt).
 \end{eqnarray*}

In order to compute the communication costs of the algorithm, we must make some assumptions on the topology of the network of the parallel machine.  We will assume that the processors are arranged in a 2D square grid and that nearest neighbor connections exist (at least within processor rows) and each processor column has a binary tree of connections.  In this way, adjacent processors in the same row can communicate at the cost of one message, and a processor can broadcast a message to all processors in its column at the cost of $O(\log P)$ messages.

Communication occurs at lines~\ref{line:diagbroadcast}, \ref{line:sendright}, \ref{line:recvleft}, and \ref{line:offdiagbroadcast}.  Since every message is of size $O(n^2 / P)$, the bandwidth cost of Algorithm~\ref{alg:ptrevc} is given by
\begin{eqnarray*}
B_{TREVC}(n,P) &=& O\lt(\frac{n^2}{P} \log P\rt) + \sqrt P \lt(O\lt(\frac{n^2}{P}\rt) + O\lt(\frac{n^2}{P} \log P\rt)\rt) \\
 &=& O\lt(\frac{n^2}{\sqrt P} \log P\rt)
 \end{eqnarray*}
 and the latency cost is
 \begin{eqnarray*}
L_{TREVC}(n,P) &=& O\lt(\log P\rt) + \sqrt P \lt(O\lt(1\rt) + O\lt(\log P\rt)\rt) \\
 &=& O\lt(\sqrt P \log P\rt).
 \end{eqnarray*}

\section{Successive Band Reduction}
\label{sec_SBR}

Here we discuss a different class of communication-minimizing algorithms, variants on the conventional reduction to tridiagonal form (for the symmetric eigenvalue problem) or bidiagonal form (for the SVD). These algorithms for the symmetric case were discussed at length in \cite{SBR1,SBR2}, which in turn refer back to algorithms originating in \cite{Rutishauser63,Schwarz68}. Our SVD algorithms are natural variations of these. These algorithms can minimize the number of words moved in an asymptotic sense on a sequential machine with two levels of memory hierarchy (say main memory and a cache of size $M$) when computing just the eigenvalues (respectively, singular values) or additionally all the eigenvectors (resp. all the left and/or right singular vectors) of a dense symmetric matrix (respectively, dense general matrix). Minimizing the latency cost, dealing with multiple levels of memory hierarchy, and minimizing both bandwidth and latency costs in the parallel case all remain open problems.

The algorithm for the symmetric case is as follows.
\begin{enumerate}
\item We reduce the original dense symmetric matrix $A$ to band symmetric form, $H = Q^TAQ$ where $Q$ is orthogonal and $H$ has bandwidth $b$. (Here we define $b$ to count the number of possibly nonzero diagonals above the main diagonal, so for example a tridiagonal matrix has $b=1$.) If only eigenvalues are desired, this step does the most arithmetic, $\frac{4}{3}n^3 + O(n^2)$ as well as $O(n^3 / M^{1/2})$ slow memory references (choosing $b\approx\sqrt M$).
\item We reduce $H$ to symmetric tridiagonal form $T = U^THU$ by successively zeroing out blocks of entries of $H$ and ``bulge-chasing.''  Pseudocode for this successive band reduction is given in Algorithm~\ref{alg:sbr}, and a scheme for choosing shapes and sizes of blocks to annihilate is discussed in detail below. This step does $O(n^2 M^{1/2})$ floating point operations and, with the proper choice of block sizes, $O(n^3 / M^{1/2})$ memory references.  So while this step alone does not minimize communication, it is no worse than the previous step, which is all we need for the overall algorithm to attain $O(n^3 / M^{1/2})$ memory references.  If eigenvectors are desired, then there is a tradeoff between the number of flops and the number of memory references required (this tradeoff is made explicit in Table~\ref{tab:SBRcounts}).
\item We find the eigenvalues of $T$ and--if desired--its eigenvectors, using the \textbf{MRRR} algorithm \cite{DPV06}, which computes them all in $O(n^2)$ flops, memory references, and space. Thus this step is much cheaper than the rest of the algorithm, and if only eigenvalues are desired, we are done.
\item If eigenvectors are desired, we must multiply the transformations from steps 1 and 2 times the eigenvectors from step 3, for an additional cost of $2n^3$ flops and $O(n^3 / \sqrt{M})$ slow memory references. 
\end{enumerate}
Altogether, if only eigenvalues are desired, this algorithm does about as much arithmetic as the conventional algorithm, and attains the lower bound on the number of words moved for most matrix sizes, requiring $O(n^3/M^{1/2})$ memory references; in contrast, the conventional algorithm \cite{lapackmanual} moves $\Omega(n^3)$ words in the reduction to tridiagonal form. However, if eigenvectors are also desired, the algorithm must do more floating point operations.  For all matrix sizes, we can choose a band reduction scheme that will attain the asymptotic lower bound for the number of words moved and increase the number of flops by only a constant factor ($2\times$ or $2.6\times$, depending on the matrix size).

\begin{algorithm}
\protect\caption{Successive band reduction of symmetric banded matrix}
\label{alg:sbr}
\begin{algorithmic}[1]
\REQUIRE $A\in\mathbb{R}^{n\times n}$ is symmetric with bandwidth $b=b_1$
	\FOR{$i=1$ to $s$}
		\STATE \COMMENT{eliminate $d_i$ diagonals from band of remaining width $b_i$}
		\FOR{$j=1$ to $(n-b_i)/c_i$}
		\STATE \COMMENT{eliminate $c_i$ columns from bottom $d_i$ diagonals}
			\STATE zero out $d_ic_i$ entries of parallelogram by orthogonal transformation (e.g. parallelograms 1 and 6 in Figure~\ref{fig:SBR_sym})
			\STATE perform two-sided symmetric update, creating bulge (e.g. $Q_1$ creates bulge parallelogram 2 in Figure~\ref{fig:SBR_sym})
			\FOR{$k=1$ to $(n-jc_i)/b_i$}
				\STATE ``chase the bulge'' (zero out just the $d_ic_i$ entries of the parallelogram) (e.g. parallelograms 2 and 3 in Figure~\ref{fig:SBR_sym})
				\STATE perform two-sided symmetric update, creating bulge \COMMENT{(e.g. $Q_3$ creates bulge parallelogram 4 in Figure~\ref{fig:SBR_sym})}
			\ENDFOR
		\ENDFOR
		\STATE $b_{i+1}=b_i-d_i$
	\ENDFOR 
\ENSURE $A$ is symmetric and tridiagonal
\end{algorithmic}
\end{algorithm}

Table~\ref{tab:SBRcounts} compares the asymptotic costs for both arithmetic and communication of the conventional direct tridiagonalization with the two-step reduction approach outlined above.  The analysis for Table~\ref{tab:SBRcounts} is given in Appendix~\ref{app:SBRcounts}.  There are several parameters associated with the two-step approach.  We let $b$ denote the bandwidth achieved after the full-to-banded step.  If we let $b=1$ we attain the direct tridiagonalization algorithm.  If we let $b=\Theta(\sqrt M)$, the full-to-banded step attains the communication lower bound.  If eigenvectors are desired, then we form the orthogonal matrix $Q$ explicitly (at a cost of $\frac43 n^3$ flops) so that further transformations can be applied directly to $Q$ to construct $QU$, where $A=(QU)T(QU)^T$.

\begin{table} \centering
  \begin{tabular}{| c | c | c |} \hline
  	& \# flops & \# words moved \\ \hline \hline
	Direct Tridiagonalization & & \\
	values & $\frac43 n^3$ & $O\lt(n^3\rt)$ \\
	vectors & $2n^3$ & $O\lt(\frac{n^3}{\sqrt M}\rt)$ \\ \hline \hline
	Full-to-banded & & \\
	values & $\frac43 n^3$ & $O\lt(\frac{n^3}{b}\rt)$\\
	vectors & $\frac43 n^3$ & $O\lt(\frac{n^3}{\sqrt M}\rt)$ \\ \hline 
	Banded-to-tridiagonal & & \\
	values & $6bn^2$ & $\displaystyle O\lt(nb_t + \sum_{i=1}^{t-1} \lt(1+\frac{d_i}{c_i}\rt)n^2 \rt)$ \\
	vectors & $\displaystyle 2\sum_{i=1}^s \frac{d_i}{b_i} n^3$ & $\displaystyle O\lt(s\;\frac{n^3}{\sqrt M}\rt)$ \\ \hline
  \end{tabular}
  \caption{Asymptotic computation and communication costs of reduction to tridiagonal form.  Only the leading term is given  in the flop count.  The cost of the conventional algorithm is given in the first block row.  The cost of the two-step approach is the sum of the bottom two block rows.  If eigenvectors are desired, the cost of each step is the sum of the two quantities given.  The parameter $t$ is defined such that $n(b_t+1)\leq M/4$, or, if no such $t$ exists, let $t=s+1$.}
  \label{tab:SBRcounts}
\end{table}

The set of parameters $s$, $b_i$, $c_i$, and $d_i$ (for $1\leq i\leq s$) characterizes the scheme for reducing the banded matrix $H$ to tridiagonal form via successive band reduction as described in \cite{SBR1,SBR2}.  Over a sequence of $s$ steps, we successively reduce the bandwidth of the matrix by annihilating sets of diagonals ($d_i$ of them at the $i^\text{th}$ step).  We let $b_1=b$ and 
$$b_i = b - \sum_{j=1}^{i-1} d_j.$$
We assume that we reach tridiagonal form after $s$ steps, or that 
$$\sum_{i=1}^s d_i = b-1.$$
Because we are applying two-sided similarity transformations, we can eliminate up to $b_i-d_i$ columns of the last $d_i$ subdiagonals at a time (otherwise the transformation from the right would fill in zeros created by the transformation from the left). We let $c_i$ denote the number of columns we eliminate at a time and introduce the constraint 
$$c_i+d_i \leq b_i.$$ 
This process of annihilating a $d\times c$ parallelogram will introduce a trapezoidal bulge in the band of the matrix which we will chase off the end of the band.  We note that we do not chase the entire trapezoidal bulge; we eliminate only the first $c_i$ columns (another $d\times c$ parallelogram) as that is sufficient for the next sweep not to fill in any more diagonals.  These parameters and the bulge-chasing process are illustrated in Figure~\ref{fig:SBR_sym}.

\begin{figure}
\centering
\includegraphics[scale=.4]{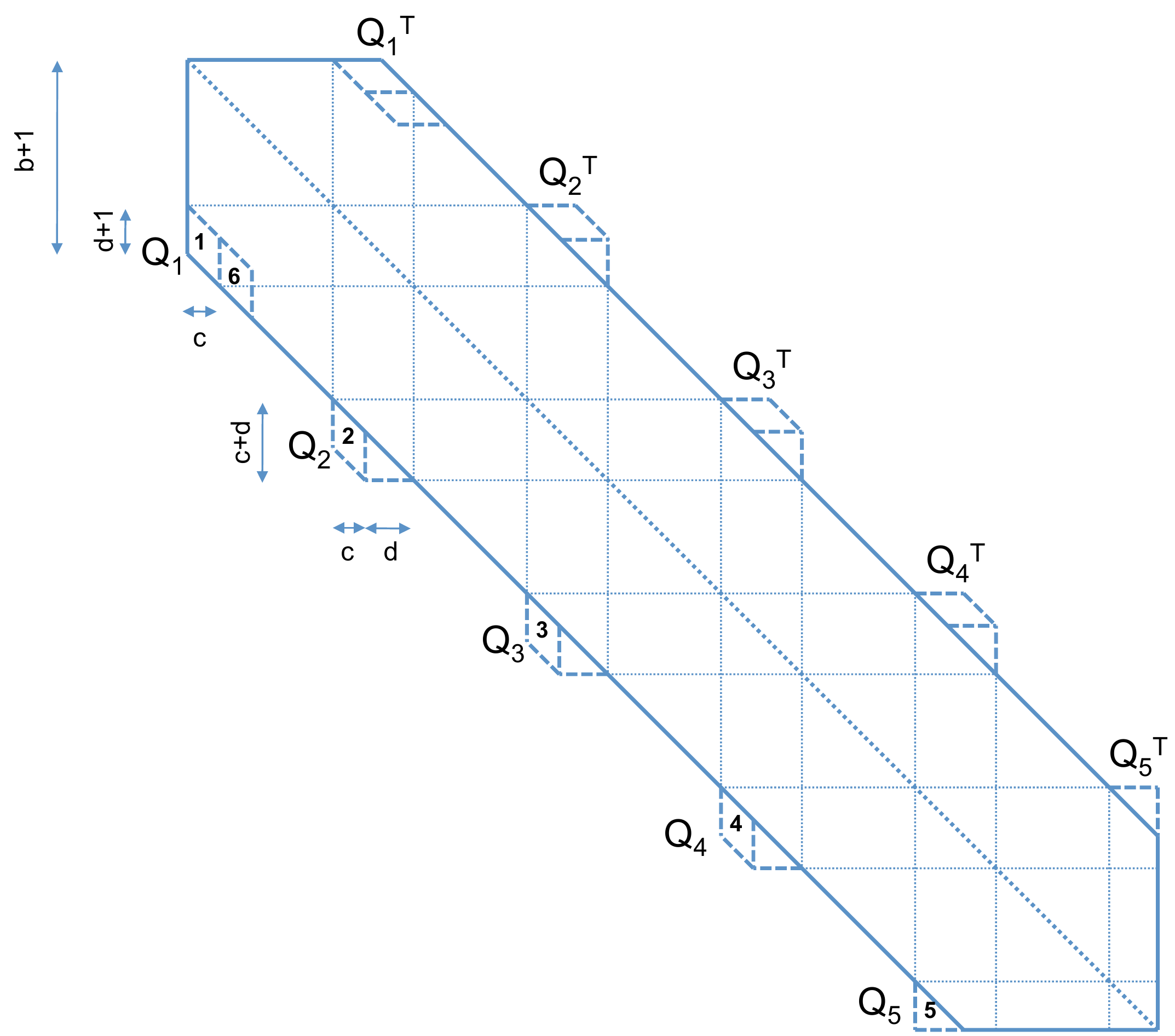}
\caption{First pass of reducing a symmetric matrix from band to tridiagonal form.  The numbers represent the order in which the parallelograms are eliminated, and $Q_i$ represents the orthogonal transformation which is applied to the rows and columns to eliminate the $i^\text{th}$ parallelogram.}
\label{fig:SBR_sym}
\end{figure}

We now discuss two different schemes (i.e. choices of $\{c_i\}$ and $\{d_i\}$) for reducing the banded matrix to tridiagonal form.\footnote{As shown in Figure~\ref{fig:SBR_sym}, for all choices of $\{c_i\}$ and $\{d_i\}$, we chase all the bulges created by the first parallelogram (numbered 1) before annihilating the adjacent one (numbered 6).  It is possible to reorder the operations that do not overlap (say eliminate parallelogram 6 before eliminating the bulge parallelogram 4), and such a reordering may reduce memory traffic.  However, for our purposes, it will be sufficient to chase one bulge at a time.}  Consider choosing $d_1=b-1$ and $c_1=1$, so that $s=1$, as in the parallel algorithm of \cite{lang}.  Then the communication cost of the banded-to-tridiagonal reduction is $O(bn^2)$.  If we choose $b=\alpha_1 \sqrt M$ for some constant $\alpha_1<1/2$, then four $b\times b$ blocks can fit into fast memory at a time, and by results in \cite{DGHL08}, the communication cost of reduction from full to banded form is $O\lt(\frac{n^3}{\sqrt M}\rt)$.  The communication cost of the reduction from banded to tridiagonal (choosing $d_1=b-1$) is $O(n^2\sqrt M)$, and so the total communication cost for reduction from full to tridiagonal form is $O\lt(\frac{n^3}{\sqrt M}+n^2\sqrt M\rt)$.  This algorithm attains the lower bound if the first term dominates the second, requiring $n=\Omega(M)$.  Thus, for sufficiently large matrices, this scheme attains the communication lower bound and does not increase the leading term of the flop count from the conventional method if only eigenvalues are desired.

If eigenvectors are desired, then the number of flops required to compute the orthogonal matrix $QU$ in the banded-to-tridiagonal step is $2n^3$ and the number of words moved in these updates is $O(n^3 / \sqrt M)$.  If $T=V\Lambda V^T$ is the eigendecomposition of the tridiagonal matrix, then $A=(QUV)\Lambda(QUV)^T$ is the eigendecomposition of $A$.  Thus, if the eigenvectors of $A$ are desired, we must also multiply the explicit orthogonal matrix $QU$ times the eigenvector matrix $V$, which costs an extra $2n^3$ flops and $O(n^3 / \sqrt M)$ words moved.  Thus the communication lower bound is still attained, and the total flop count is increased to $\frac{20}{3}n^3$ as compared to $\frac{10}{3}n^3$ in the direct tridiagonalization.

Note that if $n<M$, the $n^2\sqrt M$ term dominates the communication cost and prevents the scheme above from attaining the lower bound.  Consider values of $n$ that are larger than $\sqrt M$ (so the matrix does not fit into fast memory) but smaller than $M$ (so one or more diagonals do fit into fast memory).  In this case, we reduce the banded matrix to tridiagonal in two steps ($s=2$).  First, choose $d_1 = b - \alpha_2\frac Mn$ and $c_1 = \alpha_2\frac Mn$, where $\alpha_2<1/4$ is a constant chosen so that $\alpha_2 \frac Mn$ is an integer.  This implies that $b_2 = \alpha_2 \frac Mn$, so we choose $d_2 = \alpha_2\frac Mn -1$ and $c_2=1$, reducing to tridiagonal after two steps.

Suppose only eigenvalues are desired.  Note that $nb_2 = \alpha_2 M$, so after the first pass, the band fits in fast memory.  Note that we choose $\alpha_2<1/4$ so that the band itself fits in a quarter of fast memory, and another quarter of fast memory is available for bulges and other temporary storage.  If eigenvectors are desired, then the other half of memory can be used in the computation of $QU$.  Thus, the number of words moved in the reduction is $O(nb_2 + (1+d_1/c_1)n^2)$.  If we choose $b=\alpha_1 \sqrt M$ for some constant $\alpha_1 < 1/2$, then $\frac{d_1}{c_1}=\Theta\lt(\frac{n}{\sqrt M}\rt)$, and the total communication cost for reduction from full to tridiagonal form is
$$O\lt(\frac{n^3}{\sqrt M} + n^2 +M\rt)$$
and since $n^2>M$, the first term dominates and the scheme attains the lower bound.  Again, the change in the number of flops in the reduction compared to the conventional algorithm is a lower order term.

If eigenvectors are desired, then the extra flops required to update the orthogonal matrix $Q$ to compute $QU$ is $2(\frac{d_1}{b_1}+\frac{d_2}{b_2})n^3$.  Since $d_i < b_i$, this quantity is bounded by $4n^3$.  Since we also have to multiply the updated orthogonal matrix $QU$ times the eigenvector matrix $V$ of the tridiagonal $T$ (for a cost of another $2n^3$ flops), this increases the total flop count to no more than $\frac{26}{3}n^3$.  The extra communication cost of these updates is $O\lt(\frac{n^3}{\sqrt M}\rt)$ words, thus maintaining the asymptotic lower bound.\footnote{Updating the matrix $Q$ must be done in a careful way to attain this communication cost.  See Appendix~\ref{app:SBRcounts} for the algorithm.}

We may unify the two schemes for the two relative sizes of $n$ and $M$ as follows.  Choose
\begin{eqnarray*}
b_1 &=& \alpha_1 \sqrt M \\
b_2 &=& \max \lt\{ \alpha_2 \frac Mn, 1 \rt\} \\
d_1 &=& b_1 - b_2 \\
c_1 &=& b_2 
\end{eqnarray*}
where $\alpha_1 < 1/2, \alpha_2 < 1/4$ are constants, and, if necessary ($b_2>1$),
\begin{eqnarray*}
d_2 &=& b_2-1 \\
c_2 &=& 1.
\end{eqnarray*} 
In this way we attain the communication lower bound whether or not eigenvectors are desired.  If only eigenvalues are desired, then the two-step approach does $\frac43 n^3 + O(n^2)$ arithmetic which matches the cost of direct tridiagonalization.  If eigenvectors are also desired, the number of flops required by the SBR approach is bounded by either $\frac{20}{3}n^3$ or $\frac{26}{3}n^3$ depending on the number of passes used to reduce the band to tridiagonal, as compared to the $\frac{10}{3}n^3$ flops required by the conventional algorithm.

\begin{table} \centering
  \begin{tabular}{| c | c | c |} \hline
      $\sqrt M < n < M$ & \# flops & \# words moved \\ \hline \hline
    Full-to-banded & & \\
    $b=\alpha_1 \sqrt M$ & & \\
    values & $\frac43 n^3$ & $O\lt(\frac{n^3}{\sqrt M}\rt)$\\
    vectors & $\frac43 n^3$ & $O\lt(\frac{n^3}{\sqrt M}\rt)$ \\ \hline
    Banded-to-tridiagonal & & \\
    $d_1 = b-\alpha_2 (M/n), \; c_1 = \alpha_2 (M/n)$ & & \\
    $d_2 = b_2-1, \; c_2 = 1$ & & \\
    values & $6\alpha_1n^2\sqrt M$ & $O\lt(\frac{n^3}{\sqrt M}\rt)$ \\
    vectors & $4 n^3$ & $O\lt(\frac{n^3}{\sqrt M}\rt)$ \\ \hline
    Recovering eigenvector matrix & & \\
    vectors & $2n^3$ & $O\lt(\frac{n^3}{\sqrt M}\rt)$ \\ \hline \hline
    Total costs & & \\
    values only & $\frac43 n^3$ & $O\lt(\frac{n^3}{\sqrt M}\rt)$ \\
    values and vectors & $\frac{26}{3}n^3$ & $O\lt(\frac{n^3}{\sqrt M}\rt)$ \\ \hline
  \end{tabular}
  \caption{Asymptotic computation and communication costs of reduction to tridiagonal form in the case of $\sqrt M < n < M$.  The parameters $\alpha_1 < 1/2, \alpha_2 < 1/4$ are constants.  Only the leading term is given in the flop count.  The costs correspond to the specific choices of $b$, $d_1$, $c_1$, $d_2$, and $c_2$ listed in the first column.}
  \label{tab:SBRcase1}
\end{table}

\begin{table} \centering
  \begin{tabular}{| c | c | c |} \hline
      $n>M$ & \# flops & \# words moved \\ \hline \hline
    Full-to-banded & & \\
    $b=\alpha_1 \sqrt M$ & & \\
    values & $\frac43 n^3$ & $O\lt(\frac{n^3}{\sqrt M}\rt)$\\
    vectors & $\frac43 n^3$ & $O\lt(\frac{n^3}{\sqrt M}\rt)$ \\ \hline
    Banded-to-tridiagonal & & \\
    $d_1 = b-1, \; c_1 = 1$ & & \\
    values & $6\alpha_1n^2\sqrt M$ & $O\lt(\frac{n^3}{\sqrt M}\rt)$ \\
    vectors & $2 n^3$ & $O\lt(\frac{n^3}{\sqrt M}\rt)$ \\ \hline
    Recovering eigenvector matrix & & \\
    vectors & $2n^3$ & $O\lt(\frac{n^3}{\sqrt M}\rt)$ \\ \hline \hline
    Total costs & & \\
    values only & $\frac43 n^3$ & $O\lt(\frac{n^3}{\sqrt M}\rt)$ \\
    values and vectors & $\frac{20}{3}n^3$ & $O\lt(\frac{n^3}{\sqrt M}\rt)$ \\ \hline
  \end{tabular}
  \caption{Asymptotic computation and communication costs of reduction to tridiagonal form in the case of $n>M$.  The constant $\alpha_1$ is chosen to be smaller than $1/2$.  Only the leading term is given in the flop count.  The costs correspond to the specific choices of $b$, $d_1$, and $c_1$ listed in the first column.}
  \label{tab:SBRcase2}
\end{table} 

\section{Conclusions} 
\label{Conc}

We have presented numerically stable sequential and parallel
algorithms for computing eigenvalues and eigenvectors or the SVD, 
that also attain known lower bounds on communication, 
i.e. the number of words moved and the number of messages.

The first class of algorithms, which do {\em randomized
spectral divide-and-conquer}, do several times as much
arithmetic as the conventional algorithms, and are shown 
to work with high probability. Depending on the problem,
they either return the generalized Schur form (for regular
pencils), the Schur form (for single matrices), or the SVD.
But in the case of nonsymmetric matrices
or regular pencils whose $\epsilon$-pseudo-spectra include 
one or more large connected components in the complex plane 
containing very many eigenvalues, it is possible that the 
algorithm will only return convex enclosures of these
connected components. In contrast, the conventional algorithm 
(Hessenberg QR iteration) would return eigenvalue approximations
sampled from these components. Which kind of answer is 
more useful is application dependent. For symmetric matrices
and the SVD, this issue does not arise, and with high probability
the answer is always a list of eigenvalues.

A remaining open problem is to strengthen the probabilistic analysis
of Theorem~\ref{thm_rurv}, which we believe underestimates the
likelihood of convergence of our algorithm.

Given the Schur form of a nonsymmetric matrix or matrix pencil,
we can also compute the eigenvectors in a communication-optimal
fashion.

Our last class of algorithms uses a technique called 
{\em successive band reduction (SBR)}, which applies only to the
symmetric eigenvalue problem and SVD. SBR deterministically
reduces a symmetric matrix to tridiagonal form (or a general
matrix to bidiagonal form) after which a symmetric
tridiagonal eigensolver (resp. bidiagonal SVD) running in $O(n^2)$
time is used to complete the problem. With appropriate
choices of parameters, sequential SBR minimizes the number of 
words moved between 2 levels of memory hierarchy; minimizing
the number of words moved in the parallel case, or minimizing
the number of messages in either the sequential or parallel
cases are open problems.
SBR on symmetric matrices only does double the arithmetic operations 
of the conventional, non-communication-avoiding algorithm when
the matrix is large enough  (or 2.6x the arithmetic 
when the matrix is too large to fit in fast memory, 
but small enough for one row or column to fit).

We close with several more open problems. First, we would like to
extend our communication lower bounds so that they are not
algorithm-specific, but algorithm-independent, i.e. apply for
{\em any} algorithm for solving the eigenvalue problem or SVD.
Second, we would like to do as little extra arithmetic as possible
while still minimizing communication. Third, we would like to 
implement the algorithms discussed here and determine under which
circumstances they are faster than conventional algorithms.

\section{Acknowledgements}
This research is supported by Microsoft (Award \#024263) and Intel (Award \#024894) funding and by matching funding by U.C. Discovery (Award \#DIG07-10227). Additional support comes from Par Lab affiliates National Instruments, NEC, Nokia, NVIDIA, Samsung, and Sun Microsystems.  James Demmel also acknowledges the support of DOE Grants DE-FC02-06ER25786, DE-SC0003959, and DE-SC0004938, and NSF Grant OCI-1032639.  Ioana Dumitriu's research is supported by NSF CAREER Award DMS-0847661. She would like to thank MSRI for their hospitality during the Fall 2010 quarter; this work was completed while she was a participant in the program \emph{Random Matrix Theory, Interacting Particle Systems and Integrable Systems}.

\bibliographystyle{plain}
\bibliography{cubeig}

\appendix

\section{Analysis for Table~\ref{tab:SBRcounts}}
\label{app:SBRcounts}

\subsection{Direct Tridiagonalization}

We now provide the analysis for the communication costs presented in Table~\ref{tab:SBRcounts}.  The flop counts for direct tridiagonalization are well-known, see \cite{BCCDDDHHPSWW:96}, for example.  Because direct tridiagonalization requires $n$ symmetric matrix-vector products (each requires moving $O(n^2)$ data to access the matrix), the total number of words moved in reducing the symmetric matrix to tridiagonal is $O(n^3)$.  After this step, the orthogonal matrix whose similarity transformation tridiagonalizes $A$ is not stored explicitly, but as a set of Householder vectors.\footnote{Note that the explicit orthogonal matrix can be generated in $\frac 43 n^3$ flops if desired.}  Applying these Householder vectors to the eigenvector matrix of the tridiagonal matrix to recover the eigenvectors of $A$ can be done in the usual blocked way and thus costs $2n^3$ and $O\lt(\frac{n^3}{\sqrt M}\rt)$ extra memory references if the blocksize is chosen to be $\Theta(\sqrt M)$.  

\subsection{Counting Lemmas}

To simplify the analysis for the rest of the table, we point out two useful facts.  We also note that blocking Householder updates does introduce extra flops, but because the number of extra flops is a lower order term, we ignore them here and focus on BLAS-2 updates.

\begin{lemma}
\label{lem:A1}
The leading term of the number of flops required to apply a Householder transformation of a vector $u$ with $h$ nonzeros to $c$ columns (or rows if the applying from the right) of a matrix $A$ is $4hc$.
\end{lemma}

\begin{proof}
Let $\tilde A$ be $m\times c$ (some subset of the columns of $A$).  Applying the transformation associated with Householder vector $u$ is equivalent to overwriting the matrix $\tilde A$ with $\tilde A-\tau u (u^T\tilde A)$.  Since $u$ has $h$ nonzeros, the cost of computing $u^T\tilde A$ is $2hc$ which yields a dense row vector of length $c$.  Computing $\tau u$ costs $h$ multiplies.  Finally, $hc$ entries of $A$ must be updated with one multiplication and one subtraction each.  Thus, the total cost of the update is $4hc+h$.
\end{proof}

\begin{lemma}
\label{lem:A2}
The cost of applying a Householder transformation of a vector with $h$ nonzeros from the left and right to a symmetric matrix is the same as applying the Householder transformation from the left to a nonsymmetric matrix of the same size.
\end{lemma}

\begin{proof}
Let $A$ be $n\times n$ and symmetric, and let $u$ be a Householder vector with $h$ nonzeros.  As pointed out in \cite{lawn02}, the performing the following three computations overwrites the matrix $A$ with $(I-\tau u u^T)A(I-\tau u u^T)^T$:
\begin{eqnarray*}
y & \leftarrow & Au \\
v & \leftarrow & y-\frac12(y^Tu)u \\
A & \leftarrow & A - uv^T - vu^T
\end{eqnarray*}
The first computation costs $2hn$ flops, and the second line costs $4h+1$. Since each of the matrices $uv^T$ and $vu^T$ have $hn$ nonzeros, each entry of $A$ which corresponds to a nonzero in one of the matrices must be updated with a multiplication and a subtraction.  However, since $A$ is symmetric, we need to update only half of its entries, so the cost of the third line is $2hn$.  Thus, the total cost (ignoring lower order terms) is $4hn$ which is the same as the cost of applying the Householder transformation from one side to an $n\times n$ nonsymmetric matrix.
\end{proof}

\subsection{Reduction from Full to Banded}

The second row of Table~\ref{tab:SBRcounts} corresponds to the first step of the SBR approach, reducing a full symmetric matrix to banded form by a sequence of QR factorizations on panels followed by two-sided updates.  Algorithm~\ref{alg:sym2band} describes the reduction of a full symmetric matrix to banded form.  Here \textbf{TSQR} refers to the Tall-Skinny QR decomposition described in \cite{DGHL08}, which returns $Q$ in a factored form whose details do not concern us here. The multiplication by $Q^T$ and $Q$ in the next two lines assumes $Q$ is in this factored form and exploits the symmetry of $A$.  See \cite{LKDB10} for the details of a multicore implementation of this algorithm.  By Lemma~\ref{lem:A2}, the flop costs of the two-sided updates are the same as a one-sided update to a non-symmetric matrix.  At the $k^\text{th}$ step of the reduction, the QR is performed on a panel of size $(n-kb)\times b$, yielding $(n-kb)b-\frac12 b^2$ Householder entries which are then applied to $n-kb+b$ rows and columns.  Thus, by Lemma~\ref{lem:A1}, the arithmetic cost of the reduction is roughly
$$\sum_{k=1}^{n/b-2} 2(n-kb)b^2-\frac13 b^3 + 4 \lt((n-kb)b-\frac12 b^2\rt) (n-kb+b) = \frac43 n^3 + O(bn^2).$$
By results for \textbf{TSQR} and applying $Q$ and $Q^T$ in Appendix B of \cite{DGHL08}, all of these flops may be organized in a blocked fashion to achieve a communication cost of $O\lt(\frac{n^3}{b}\rt)$ words, where $b<\sqrt M / 2$.  Note that $b<\sqrt M / 2$ implies that $4b^2\leq M$, or four $b\times b$ blocks fit into fast memory at the same time.

\begin{algorithm}
\protect\caption{Reduction of $A=A^T$ from dense to band form with bandwidth $b$}
\label{alg:sym2band}
\begin{algorithmic}[1]
\REQUIRE $A\in\mathbb{R}^{n\times n}$ is symmetric, and assume that $b$ divides $n$ for simplicity
	\FOR{$i=1$ to $n-2b+1$ step $b$}
		\STATE $[Q,R] = \text{TSQR}(A(i+b:n,i:i+b-1))$
		\STATE $A(i+b:n,:) = Q^T \cdot A(i+b:n,:)$ 
		\STATE $A(:,i+b:n) = A(:,i+b:n) \cdot Q$ 
	\ENDFOR
\ENSURE $A$ is symmetric and banded with bandwidth $b$
\end{algorithmic}
\end{algorithm}

If eigenvectors are desired, the orthogonal matrix $Q$ that reduces the full matrix to banded form via similarity transformation is formed explicitly.  In this way, the many small Householder transformations of the banded-to-tridiagonal step can be applied to $Q$ as they are computed and then discarded.  The arithmetic cost of forming $Q$ explicitly is that of applying the Householder transformations to the identity matrix in reverse order as they were computed.\footnote{This method mirrors the computation performed in LAPACK's {\tt xORGQR} for generating the corresponding orthogonal matrix from a set of Householder vectors.}  Since the number of Householder entries used to reduce the second-to-last column panel is roughly $\frac12 b^2$ and the number of rows on which these transformations operate is $b$, the transformations need be applied only to the last $b$ columns of the identity matrix.  The next set of transformations, computed to reduce the third-to-last column panel of $A$, operate on the last $2b$ rows.  Thus, they need to be applied to only the last $2b$ columns of the identity matrix.   Since the number of Householder entries used to reduce the $i^\text{th}$ column panel is roughly $(n-ib)b-\frac12 b^2$ and the number of columns to be updated is $(n-ib)$, by Lemma~\ref{lem:A1}, the total number of flops for constructing $Q$ explicitly is given by
$$\sum_{i=1}^{n/b-1}4\lt((n-ib)b-\frac12 b^2\rt)(n-ib) = \frac43n^3 + O(bn^2).$$
Again, as shown in \cite{DGHL08}, the application of these Householder entries can be organized so that the communication cost is $O\lt(\frac{n^3}{b}\rt)$ words where $b\leq\sqrt M/2$.

\subsection{Reduction from Banded to Tridiagonal}

Some analysis for the third row of Table~\ref{tab:SBRcounts} appears in \cite{SBR1} for the flop counts in the case $c_i=1$ for each $i$.  We reproduce their results for general choices of $\{c_i\}$ and also provide analysis for the cost of the data movement.  We will re-use the terminology from \cite{SBR1} and refer the reader to Figure 2(b) in that paper for a detailed illustration of one bulge-chase.

Bulge-chasing can be decomposed into four operations: ``QR'' of a $(d+c)\times c$ matrix, ``Pre''-multiplication of an orthogonal matrix which updates $d+c$ rows, a two-sided ``Sym''mmetric update of $(d+c)\times (d+c)$ submatrix, and ``Post''-multiplication of the transpose of the orthogonal matrix which updates $d+c$ columns.  We can count the arithmetic cost in terms of applying Householder vectors to sets of rows/columns (using Lemmas~\ref{lem:A1} and \ref{lem:A2}).  Consider eliminating column $k$ of the parallelogram.  In the QR step, the associated Householder transformation must be applied to the remaining $c-k$ columns.  In the Pre step, the transformation is applied to $b-c$ columns.  In the Sym step, the two-sided update is performed on a $(d+c)\times (d+c)$ submatrix.  Finally, in the Post step, the transformation is applied from the right to $b-(c-k)$ rows.  Since every Householder vector used to eliminate the parallelogram has $d$ nonzeros, the cost of eliminating each column of the parallelogram is $4d(2b+d)$, and the cost of eliminating the entire parallelogram ($c$ columns) is $8bcd+4cd^2$.

Bulge-chasing pushes the bulge $b$ columns down the band each time.  Thus, the number of flops required to eliminate a set of $d$ diagonals is about
$$\sum_{j=1}^{n/c} \lt(8bcd+4cd^2\rt) \frac{n-jc}{b} = (4d + 2\frac{d^2}{b})n^2.$$
Since $d<b$, this cost is bounded above by $6dn^2$.  Taking $s$ steps to eliminate all the diagonals, we see that the total arithmetic cost of the banded-to-tridiagonal reduction (ignoring lower order terms) is
$$\sum_{k=i}^s 6 d_i n^2 = 6bn^2.$$

We now consider the data movement.  First, suppose the band is small enough to fit into memory into one fourth of memory; that is, $n(b+1)<M/4$ (we pick the fraction $1/4$ so that the band fits in a quarter of the memory, the extra storage required for bulge chasing fits in a quarter of the memory, and the other half of memory can be reserved for the vector updating process discussed below).  Then the communication cost of the reduction is simply reading in the band: $nb$ words.  Suppose the band is too large to fit into memory.  In this case, since we chase each bulge all the way down the band, we will obtain very little locality.  Even though we eliminate a bulge as soon as it's been created, we ignore the fact that some data (including the bulge itself) fits in memory and will be reused (it is a lower order term anyway).  Consider the number of words accessed to chase one bulge: QR accesses $cd+\frac12 c^2$ words, Pre accesses $(b-c)(c+d)$ words, Sym accesses $\frac12(c+d)^2$ words, and Post accesses $b(c+d)-\frac12c^2$ words.  The total number of words accessed for one bulge-chase is then $2b(c+d)+\frac12(c+d)^2 - c^2$.  Using the same summation as before, the communication cost of chasing all the bulges from eliminating $d$ diagonals (ignoring lower order terms) is
$$\sum_{j=1}^{n/c} \lt(2b(c+d)+\frac12(c+d)^2 - c^2\rt) \frac{n-jc}{b} = \lt(1+\frac dc - \frac{c}{2b} + \frac{(c+d)^2}{4bc}\rt)n^2.$$
Since $c+d\leq b$, and by ignoring the negative term, we obtain a communication cost for eliminating $d$ diagonals of $O\lt(\lt(1+\frac dc\rt)n^2\rt)$.

In order to obtain the communication cost for the entire banded-to-tridiagonal step, we sum up the costs of all the reductions of $\{d_i\}$ diagonals until the band fits entirely in memory.  That is, the number of words moved in the banded-to-tridiagonal step is
$$O\lt(nb_t + \sum_{i=1}^{t-1} \lt(1+\frac{d_i}{c_i}\rt)n^2 \rt)$$
where $n(b_t+1) \leq M$.  Note that if no such $t$ exists (in the case $2n>M$), the number of words is given by
$$O\lt(\sum_{i=1}^{s} \lt(1+\frac{d_i}{c_i}\rt)n^2 \rt)$$
where $b_{s+1}=1$.

If eigenvectors are desired, each of the orthogonal transformations used to eliminate a parallelogram or chase a bulge must also be applied to the orthogonal matrix $Q$ which is formed explicitly in the full-to-banded step.  Note that the transformations must be applied to $Q$ from the right (we are forming the matrix $QU$ such that $(QU)^TAQU=T$ is tridiagonal).  However, these updates are one-sided and since many of the orthogonal transformations work on mutually disjoint sets of columns, we can reorder the updates to $Q$.  The restrictions on this reordering are made explicit in \cite{prism:17}, and we use the same reordering proposed in that paper.  Note that to be able to reorder the updates on the eigenvector matrix $Q$, we must temporarily store a set of Householder entries which are computed in the band reduction process.

Suppose we save up all the Householder entries from eliminating $k$ parallelograms and chasing their associated bulges completely off the band.  Since every Householder transformation consists of $cd$ Householder entries, saving up the transformations after eliminating $k$ parallelograms requires $O\lt(kcd\frac nb\rt)$ words of extra memory.  

Following the notation of \cite{prism:17}, we number the Householder transformations by the ordered pair $(i,j)$ where $i$ is the number of the initial parallelogram annihilated and $j$ is the number of the bulge chased by the transformation (so $1\leq i\leq k$ and $1\leq j\leq n/b$).  For example, in Figure~\ref{fig:SBR_sym}, $Q_1$ would we numbered $(1,1)$, $Q_3$ would be numbered $(1,3)$, and the orthogonal transformation annihilating parallelogram $6$ would be numbered $(2,1)$.  We will group these Householder transformations not by their initial parallelogram number but by their bulge number (the second index in the ordered pair).  As argued in \cite{prism:17}, we can apply these updates in groups in reverse order, so that all $(i,4)$ transformations will be applied before any $(i,3)$ transformations.  As long as the order of the updates within the groups is done in increasing order with respect to the initial parallelogram number (e.g. $(3,4)$ before $(4,4)$), we maintain correctness of the algorithm.

Consider the group of transformations with bulge number $j$.  The transformation $(1,j)$ eliminates a parallelogram of size $d\times c$ and therefore updates $d+c$ columns of $Q$.  Since the transformation $(2,j)$ eliminates the adjacent parallelogram, it updates an overlapping set of columns.  The first (leftmost) $d$ columns are affected by $(1,j)$, but the last (rightmost) $c$ columns are disjoint.  Since all of the parallelograms are adjacent, the total number of columns updated by the $k$ transformations with bulge number $j$ is $d+kc$.

We choose $k$ such that $kc=\Theta(\sqrt M)$.  Note that we may always choose $k$ in this way if we assume the initial bandwidth is chosen to be $b=O(\sqrt M)$ since $c<b$ (this is a safe assumption because choosing a larger bandwidth does not reduce memory traffic in the full-to-banded step).  This implies that since $d<b$, the extra memory required is $O(n\sqrt M)$ words, each group of Householder transformations consists of $kcd=O(M)$ words, and that the number of columns affected by each group of Householder transformations is $O(\sqrt M)$.

We now describe the updates to $Q$.  We assume that all of the Householder entries for $k$ sweeps of the reduction process have been stored in slow memory and that $Q$ resides in slow memory.\footnote{In the communication cost for the reduction of the band to tridiagonal form, we assumed that once the band fits in memory it may stay there until the algorithm completes.  If eigenvectors are desired, then this update process must occur with the band itself residing in fast memory.  We choose constants such that half of fast memory is reserved for storing the band and the other half of fast memory may be used for the updating process.}  The updating process is given as Algorithm~\ref{alg:vectorupdate}.   

\begin{algorithm}
\protect\caption{Updating eigenvector matrix $Q$ with Householder transformations from $k$ sweeps}
\label{alg:vectorupdate}
\begin{algorithmic}[1]
\REQUIRE $Q\in\mathbb{R}^{n\times n}$ is orthogonal, transformations $U(i,j)$ have been computed for $1\leq i\leq k$ and $1\leq j\leq J$, and assume that $B$ divides $n$ for simplicity
	\FOR{$l=1$ to $n-B$ step $B$}
		\FOR{$j=J$ down to $1$}
			\STATE read $U(i,j)$ for $1\leq i\leq k$
			\STATE read columns of row panel $Q(i:i+B-1,:)$ (not already in fast memory) effected by $U(1:k,j)$ 
			\FOR{$i=1$ to $k$}
				\STATE apply $U(i,j)$
			\ENDFOR
		\ENDFOR 
	\ENDFOR
\ENSURE $Q$ is orthogonal
\end{algorithmic}
\end{algorithm}

Let $J=O(n/b)$ be the largest bulge number. We update $Q$ one row panel at a time in reverse order by the bulge number $j$ and then in increasing order by sweep number $i$.  Since the size of each group of transformations is $O(M)$, and the number of rows and columns of $Q$ accessed at each iteration are both $O(\sqrt M)$, all of the necessary information resides in fast memory during the application of the $(i,j)$ transformations.  The columns effected by transformations $(:,j)$ and $(:,j-1)$ may overlap, but since the rightmost column effected by block $j$ is to the right of the rightmost column effected by $j-1$ for each $j$, applying the blocks of transformations in this order ensures that the entries of each row panel of $Q$ are read from and written to slow memory no more than once.  However, each Householder entry in a transformation is read once for each row panel updated.  Thus, the number of words accessed to/from slow memory for the vector update after $k$ sweeps, under the choices $kc=\Theta(\sqrt M)$ and $B=\Theta(\sqrt M)$ is 
$$\frac nB\lt(kcd\frac nb+nB\rt) = O(n^2).$$
The number of times this updating process must be done to eliminate $d$ diagonals is $\frac{n}{kc}$, and thus the total communication cost is $O\lt(\frac{n^3}{\sqrt M}\rt)$.  If $s$ is the number of steps taken to reach tridiagonal form, the total extra number of words moved in order to update the vector matrix in the banded-to-tridiagonal step is 
$$O\lt(s\;\frac{n^3}{\sqrt M}\rt).$$

After the matrices $QU$ and $T$ are computed, we use \textbf{MRRR} to find the eigenvalues and (if desired) the eigenvectors of $T$.  In order to compute the eigenvectors of $A$ we must multiply the matrix $QU$ times the eigenvector matrix of $T$, and doing so costs $2n^3$ flops and $O(n^3/\sqrt M)$ memory references.

\section{SVD via SBR}
\label{app:SVD}

The situation for the SVD is analogous to the symmetric eigenvalue problem using symmetric band reduction:
\begin{enumerate}
\item We reduce the matrix $A$ to upper triangular band form $S = U^TAV$ with bandwidth $b \approx M^{1/2}$, and $U$ and $V$ are orthogonal.
\item We reduce $S$ to upper bidiagonal form $B = W^TSX$, where $W$ and $X$ are orthogonal. We again do this by successively zeroing out blocks of
entries of $S$ of carefully chosen sizes.
\item We find the singular values of $B$, and its left and/or right singular vectors if desired. There are a number of algorithms that can compute
just the singular values in $O(n^2)$ flops and memory references, in which case we are done. But computing the singular vectors as well is more
complicated: Extending \textbf{MRRR} from the symmetric tridiagonal eigenproblem to the bidiagonal SVD has been an open problem for many years,
but an extension was recently announced \cite{Willems09}, whose existence we assume here.
\item If singular vectors are desired, we must multiply the transformations from steps 1 and 2 times the singular vectors from step 3, for an
additional cost of $O(n^3)$ flops and $O(n^3 / \sqrt{M})$ slow memory references.
\end{enumerate}

The algorithm and analysis for the SVD are entirely analogous to the symmetric case.  Algorithm~\ref{alg:full2bandbidiag} shows the first part of
the reduction to bidiagonal form.

\begin{algorithm}
\protect\caption{Reduction of $A$ from dense to upper triangular band form with bandwidth $b$}
\label{alg:full2bandbidiag}
\begin{algorithmic}[1]
\REQUIRE $A\in\mathbb{R}^{n\times n}$, and assume that $b$ divides $n$ for simplicity
    \FOR{$i=1$ to $n-2b+1$ step $b$}
        \STATE $[Q_c,R_c] = TSQR(A(i:n,i:i+b-1))$
        \STATE $A(i:n,:) = Q_c^T \cdot A(i:n,:)$
        \STATE $[Q_r,R_r] = TSQR((A(i:i+b-1,i+b:n))^T)$
        \STATE $A(:,i+b:n) = A(:,i+b:n) \cdot Q_r$
    \ENDFOR
\ENSURE $A$ is banded upper triangular with bandwidth $b$
\end{algorithmic}
\end{algorithm}

\begin{figure}
\centering
\includegraphics[scale=.5]{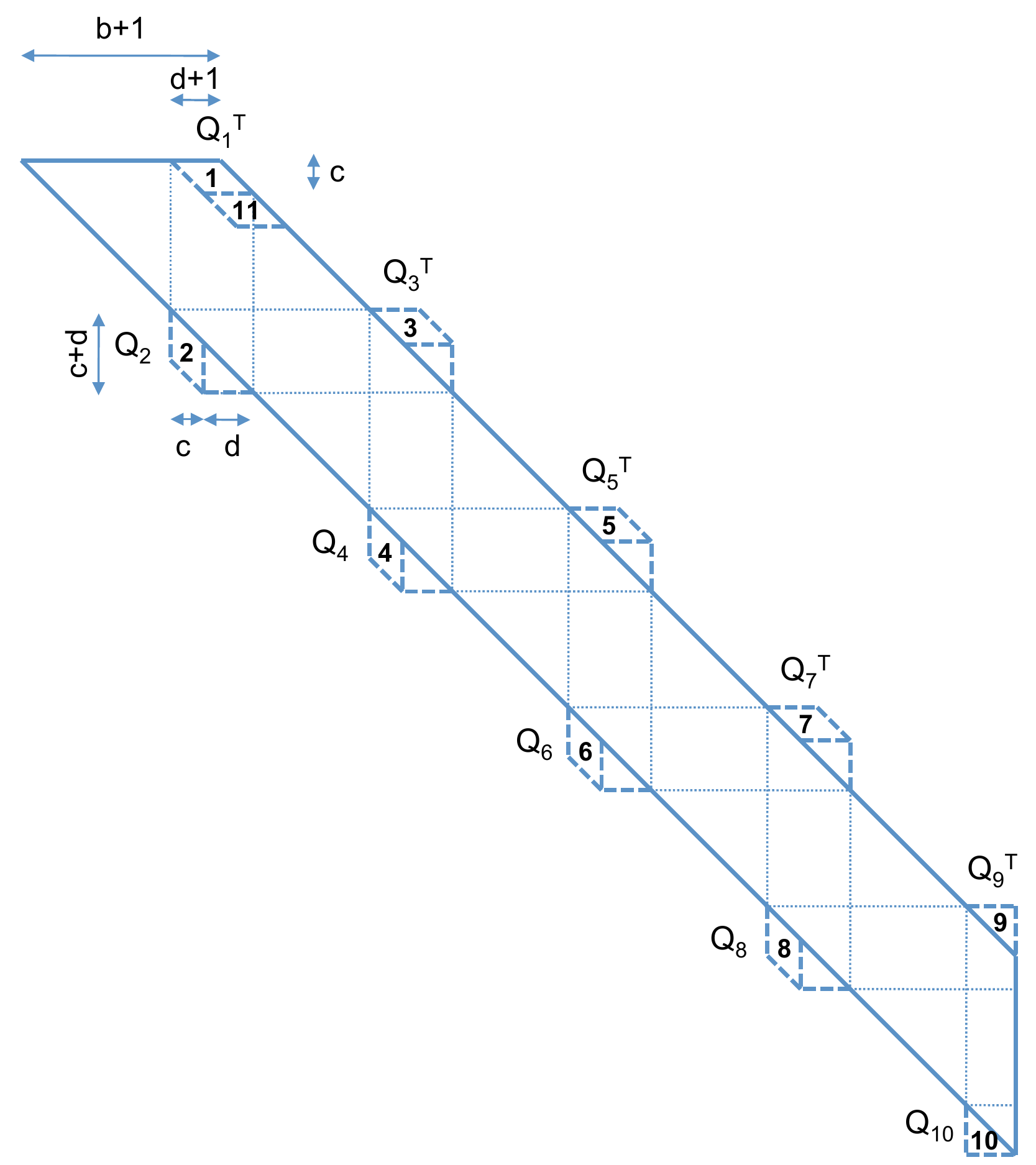}
\caption{First pass of reducing a triangular matrix from banded to bidiagonal form}
\label{fig:SBR_svd}
\end{figure}

The first pass of the second part of the SVD algorithm, reducing the upper triangular band matrix to bidiagonal form, is shown in
Figure~\ref{fig:SBR_svd}.  Table~\ref{tab:SBRSVD} provides the asymptotic computation and communication complexity of computing the SVD for the
case when only singular values are desired and the case when both left and right singular vectors as well as the singular values are desired.  The
analysis for Table~\ref{tab:SBRSVD} is analogous to the symmetric case.  Except in the case of the reduction from banded to bidiagonal form, the
leading term in the flop count at each step doubles as compared to the symmetric eigenproblem.

\begin{table} \centering
  \begin{tabular}{| c | c | c |} \hline
& \# flops & \# words moved \\ \hline \hline
    Direct Bidiagonalization & & \\
    values & $\frac83 n^3$ & $O\lt(n^3\rt)$ \\
    vectors & $4n^3$ & $O\lt(\frac{n^3}{\sqrt M}\rt)$ \\ \hline \hline
    Full-to-banded & & \\
    values & $\frac83 n^3$ & $O\lt(\frac{n^3}{b}\rt)$\\
    vectors & $\frac83 n^3$ & $O\lt(\frac{n^3}{\sqrt M}\rt)$ \\ \hline
    Banded-to-tridiagonal & & \\
    values & $8bn^2$ & $\displaystyle O\lt(nb_t + \sum_{i=1}^{t-1} \lt(1+\frac{d_i}{c_i}\rt)n^2 \rt)$ \\
    vectors & $\displaystyle 4\sum_{i=1}^s \frac{d_i}{b_i} n^3$ & $\displaystyle O\lt(s\;\frac{n^3}{\sqrt M}\rt)$ \\ \hline
    Recovering singular vectors & & \\
    vectors & $4n^3$ & $\displaystyle O\lt(\frac{n^3}{\sqrt M}\rt)$ \\ \hline
  \end{tabular}
  \caption{Asymptotic computation and communication costs of computing the SVD of a square matrix.  Only the leading term is given in the flop
count.  The cost of the conventional algorithm is given in the first block row.  The cost of the two-step approach is the sum of the bottom three
block rows.  If both sets of singular vectors are desired, the cost of each step is the sum of the two quantities given.  The parameter $t$ is
defined such that $n(b_t+1)\leq M/4$, or, if no such $t$ exists, let $t=s+1$.}
  \label{tab:SBRSVD}
\end{table}

In order to attain the communication lower bound, we use the same scheme for choosing $b$, $\{c_i\}$, and $\{d_i\}$ as in the symmetric case. That is, for large matrices ($n>M$), we reduce the band to bidiagonal in one sweep, and for smaller matrices ($\sqrt M < n < M$), we use the first sweep to reduce the size of the band so that it fits in fast memory and the second sweep to reduce to bidiagonal form.  As in the case of the symmetric eigenproblem, if only singular values are desired, the leading term in the flop count is the same as for the conventional, non-communication-avoiding algorithm.  If both left and right singular vectors are desired, for large matrices, we increase the leading term in the flop count by a factor of $2$, and for smaller matrices, we increase the leading term in the flop count by a factor of $2.6$.  The conventional approach costs $\frac83 n^3$ flops if only singular values are desired and $\frac{20}{3}n^3$ flops if all vectors and values are desired. Computing the SVD via SBR costs $\frac83 n^3$ flops if only singular values are required and either $\frac{40}{3}n^3$ or $\frac{52}{3}n^3$ flops if all vectors and values are desired.

\end{document}